%% file: BergIlchTren210707.tex
\documentclass{imamci}

\jno{dnaaxxx}
\input{MPtopbook}
\newenvironment{smallarray}[1]
 {\null\,\vcenter\bgroup\scriptsize
  \renewcommand{\arraystretch}{0.7}%
  \arraycolsep=.13885em
  \hbox\bgroup$\array{@{}#1@{}}}
 {\endarray$\egroup\egroup\,\null}

\numberwithin{equation}{section}

\newcommand{\myblue}[1]{{\color{blue}#1}}

\begin{document}

%%%%%%%%%%%%%%%%%
\title{Quasi feedback forms for differential-algebraic systems}
\author{ {\sc Thomas Berger}\\[2pt]
Institut f\"ur Mathematik, Universit\"at Paderborn, Warburger Str.~100, 33098~Paderborn, Germany, {\tt thomas.berger@math.upb.de}\\[6pt]
{\sc Achim Ilchmann}\\[2pt]
Institut f\"ur Mathematik, Technische Universit\"{a}t Ilmenau,
                          Weimarer Stra{\ss}e~25, 98693 Ilmenau,
            Germany, {\tt achim.ilchmann@tu-ilmenau.de}\\[6pt]
{\sc Stephan Trenn}\\[2pt]
Systems, Control and Applied Analysis -- Bernoulli Institute, University of Groningen, Nijenborgh~9, 9747~AG~Groningen,
            The Netherlands, {\tt s.trenn@rug.nl}\vspace*{6pt}}

\markboth{\sc T. Berger et al.}{\sc Quasi feedback-forms for differential-algebraic systems}
\pagestyle{headings}

\maketitle

%%%%%%%%%%%%%%%%%abstract style
%Two grouping braces are necessary in abstract environment
%first argument contains abstract text; second argument contains keywords
%text

\begin{abstract}
{{
We investigate feedback forms for linear time-invariant systems described by differential-algebraic
equations.
Feedback forms are representatives of certain equivalence classes.
For example  state space transformations, invertible transformations from the
left, and proportional state feedback constitute an equivalence relation.
The representative of  such an equivalence class,
which we call proportional feedback form for the above example,
 allows to read off
relevant system theoretic properties.
Our main contribution is to derive a
\textit{quasi} proportional feedback form.
This form is advantageous since it provides some geometric insight and is simple to compute, but
still allows to read off  the relevant structural properties of the control system.
We also derive a \textit{quasi} proportional and derivative feedback form.
Similar  advantages hold.
}}

%Keywords
%{keyword 1, keyword2.}
{{Differential-algebraic systems, descriptor systems, feedback forms, Wong sequences}}
\end{abstract}
%%%%%%%%%%%%%%%%%%%%%%%%%%%%%%%

\noindent
\textbf{Dedication}
\\
We dedicate  the  present note to the  memory of Nicos Karcanias~--
a friend  and  colleague.
Nicos  has had a fundamental impact in diverse areas of
\emph{systems and control theory},
in particular in
\emph{matrix pencil theory of linear systems}.
He has had a special interest in \textit{system structure},
leading to invariants and canonical forms.
See for example the very early paper~\cite{KarcKouv79}
and many more.
Our present note is in this spirit. Some of our results are closely
related to his  seminal work~\cite{LoisOzca91}.
We will refer to  this in due place.

\section{Introduction}
%\footnote{\color{blue}Nur eine Kleinigkeit, die ich an  dieser Stelle vermerke: Die Titel  der Arbeit~\cite{Lewi86} geh\"ort  klein geschrieben.}
%\footnote{\color{blue}Den  Titel  der Arbeit finde  ich  nicht \"uberzeugend, weil es ja kein ``quasi-feedback'' ist. Quasi bezieht sich  auf die kanonische Form. Dann aber m\"u\ss ten  wir die canonical form und  die quasi canonical form im  Artikel einführen. Eine m\"ogliche  \"Uberschrift w\"are dann: `Canonical forms for differential-algebraic systems', wobei ich quasi weggelassen habe, weil man  es einklammern m\"u\ss  te, und  das sieht bl\"od aus.}
%%%%%%%%%%%%%%%%%%%%%%%%%%%%%%%%%

We study \emph{structured} matrix pencils of the form~$s[E,0] - [A,B]$ with $E,A\in\R^{\ell\times n}$ and $B\in\R^{\ell\times m}$ for which we write $[E,A,B]\in \Sigma_{\ell,n,m}$. Such pencils are typically associated with differential-algebraic control systems of the form
\begin{equation}\label{eq:EAB}
 E \dot x(t) = Ax(t)+Bu(t).
\end{equation}

An essential difference between  differential-algebraic systems~\eqref{eq:EAB}
and ordinary differential systems (by this  we mean~\eqref{eq:EAB}
with  square and invertible~$E\in\R^{n\times n}$)
is  their  solution  behaviour (see~\cite{Tren13a})
and  the different  controllability concepts (see~\cite{BergReis13a}).
%For~(i) see~\cite{Tren13a} and  for~(ii) see~\cite{BergReis13a}. %\footnote{\color{blue}Wir  sollten complete  controllability und dessen algebraische Charakterisierung schon im weiteren  Text definieren, oder?}
To address various  control problems of    systems~\eqref{eq:EAB},
the well known
(quasi) Kronecker form (cf.~\cite{Kron90,Gant59d,BergTren12})
of the augmented pencil~$s[E,0] - [A,B]$
is not ``good enough''. A finer structure, which takes into
account the input matrix~$B$,  is  required.
This is achieved by     ``(quasi) canonical'' (feedback) forms.
%\footnote{\color{blue} Bis jetzt ist canonical form  im Artikel nicht zitiert (W\"urde ich auch nicht  machen.), wird aber als Begriff st\"andig verwendet. Das m\"u\ss  te man noch ausmerzen.}

A very early contribution in this spirit is  by Nicos   Karcanias and coworkers~\cite{LoisOzca91};  we will explain their  achievements in due course.
Our approach  to ``(quasi) canonical'' forms is  via
 augmented Wong sequences; this tool  is fundamental and introduced in Section~\ref{Sec:AugWong}.

In Section~\ref{Ssec:PFF}, we allow  for state space transformations of~\eqref{eq:EAB},
 invertible transformations from the  left, and proportional state feedback,
where the  latter means  to add the algebraic relation $u(t) = F_Px(t) + v(t)$ to the system~\eqref{eq:EAB}  for~$F_P\in\R^{m\times n}$ and~$v$ as new input. These transformations  constitute an equivalence relation and
the representatives of the equivalence classes are called P-feedback forms.
We derive a \textit{quasi} P-feedback form which~-- when compared to
the previous form~-- has the advantages that it provides some geometric insight and is simpler to compute,
but still allows to read off the most relevant structural properties of the control system.

In Section~\ref{Ssec:PDF-form},
 we extend the class of allowed transformations by also considering derivative feedback, i.e.\ $u(t) = F_P x(t) + F_D \dot{x}(t) + v(t)$. Again, we derive a \textit{quasi} form for the corresponding equivalence relation.

%============================================================================
\section{Augmented Wong sequences} \label{Sec:AugWong}
%============================================================================
%

Wong sequences have been introduced as a fundamental geometric tool for the analysis of matrix pencils and the derivation of quasi canonical forms~-- see~\cite{BergIlch12a,BergTren12,BergTren13}.
This approach has been extended to control systems $[E,A,B]\in\Sigma_{\ell,n,m}$ to derive a Kalman controllability decomposition~-- see~\cite{BergReis13a,BergTren14}.  Compared to matrix pencils $sE-A\in\R[s]^{\ell\times n}$, the augmented pencil $[sE-A,-B]$ with $B\in\R^{\ell\times m}$ contains additional independent variables, which are typically associated with the input of the control system~\eqref{eq:EAB}\footnote{Of course, input constraints may be present, but we ignore this for purpose of motivation.}. We like to emphasize that the augmented Wong sequences are projections of the Wong sequences corresponding to the augmented matrix pencil $s[E,0]-[A,B]$ as shown in Proposition~\ref{Prop:AWS_properties}.

The augmented Wong limits are related to the concepts of reachable and controllable spaces for the~DAE control system $[E,A,B]\in\Sigma_{\ell,n,m}$. These spaces are some of the most important notions for~(DAE) control systems and have been considered in~\cite{Lewi86} for regular systems. Further usage of these concepts can be found in the following: in~\cite{OzcaLewi89} generalized reachable and controllable subspaces of regular systems are considered; Eliopoulou and Karcanias~\cite{ElioKarc95} consider reachable and almost reachable subspaces of general~DAE systems; Frankowska~\cite{Fran90} considers the reachable subspace in terms of differential inclusions. However, to the best of our knowledge, the interplay between the (augmented) Wong-sequences and (quasi) canonical forms has not been investigated so far and
the  present contribution aims to close this gap.

%============================================================================
%\subsection{Augmented Wong sequences}\label{Ssec:AWS}
%============================================================================

\begin{Definition}[Augmented Wong sequences]\label{Def:AWS}
Let $[E,A,B]\in \Sigma_{\ell,n,m}$.
The sequences $(\cV_{[E,A,B]}^i)_{i\in\N_0}$ and $(\cW_{[E,A,B]}^i)_{i\in\N_0}$  defined as\footnote{Note that in this work $A^{-1}$ and $E^{-1}$ denote the preimage of the respective space under the induced linear map}, and not the matrix inverse (whose existence is not assumed here; in fact, $E$ and $A$ are not even assumed to be square).
\[
   \begin{aligned}
   \cV_{[E,A,B]}^0&:=\R^n,&\quad \cV_{[E,A,B]}^{i+1}&:=A^{-1}(E\cV_{[E,A,B]}^i + \im  B)\subseteq \R^n,\\
   \cW_{[E,A,B]}^0&:=\{0\},&\quad \cW_{[E,A,B]}^{i+1}&:=E^{-1}(A\cW_{[E,A,B]}^i + \im  B)\subseteq \R^n,
   \end{aligned}
\]
are called \emph{augmented Wong sequences} and
\[
   \cV^*_{[E,A,B]}  :=  \bigcap_{i\in\N_0} \cV_{[E,A,B]}^i,\qquad
    \cW^*_{[E,A,B]} := \bigcup_{i\in\N_0} \cW_{[E,A,B]}^i,
\]
are called the \emph{augmented Wong limits}.
\end{Definition}

%Compared to the Wong sequences for matrix pencils in Definition~\ref{Def11:Wong-seqs}, in the definition of the sequences $(\cV_{[E,A,B]}^i)_{i\in\N_0}$ and $(\cW_{[E,A,B]}^i)_{i\in\N_0}$ the image of~$B$ is added before taking the pre-image. If~$B=0$, then we end up with the Wong sequences. However, if $B\neq 0$, then $(\cV_{[E,A,B]}^i)_{i\in\N_0}$ and $(\cW_{[E,A,B]}^i)_{i\in\N_0}$ are in general not Wong sequences corresponding to any matrix pencil.

We highlight some important properties of the above defined sequences.

\begin{Proposition}[Properties of augmented Wong sequences]\label{Prop:AWS_properties}
Consider $[E,A,B]\in\Sigma_{\ell,n,m}$ with their augmented Wong sequences $(\cV_{[E,A,B]}^i)_{i\in\N_0}$ and $(\cW_{[E,A,B]}^i)_{i\in\N_0}$. Then we have:
\begin{enumerate}[(a)]
\item
The sequences are nested and terminate, i.e., there exist~$i^*,j^*\leq n$ such that, for all~$i,j\in\N$,
\begin{subequations}\label{eq:AWS_nested}
\begin{align}
&\cV^0_{[E,A,B]} \supsetneq \cV^1_{[E,A,B]} \supsetneq \cdots   \supsetneq \cV^{i^*}_{[E,A,B]} = \cV^{i^*+i}_{[E,A,B]}
= \cV^*_{[E,A,B]} = A^{-1}(E\cV^*_{[E,A,B]} + \im  B),\label{eq:AWS_nested1}\\
 &\cW^0_{[E,A,B]} \subsetneq \cW^1_{[E,A,B]} \subsetneq \cdots  \subsetneq \cW^{j^*}_{[E,A,B]} = \cW^{j^*+j}_{[E,A,B]}
= \cW^*_{[E,A,B]} = E^{-1}(A\cW^*_{[E,A,B]} + \im  B),\label{eq:AWS_nested2}
\end{align}
\end{subequations}
and hence their limits are well defined.

\item
The augmented Wong limits $\cV^*_{[E,A,B]},\cW^*_{[E,A,B]}\subseteq\R^n$ are linear subspaces and satisfy
\begin{equation}\label{eq:AWS_invariance}
    \begin{aligned}
    E \cW^*_{[E,A,B]} &\subseteq A\cW^*_{[E,A,B]} + \im B,\\
    A\cV^*_{[E,A,B]} &\subseteq E\cV^*_{[E,A,B]} + \im B,\\
%    E(\cV^*_{[E,A,B]} \cap \cW^*_{[E,A,B]}) &\subseteq E\cV^*_{[E,A,B]}\cap (A\cW^*_{[E,A,B]}+\im B), \\
%    &\subseteq (E\cV^*_{[E,A,B]} + \im B) \cap (A\cW^*_{[E,A,B]}+\im B), \\
%    A(\cV^*_{[E,A,B]} \cap \cW^*_{[E,A,B]}) &\subseteq (E\cV^*_{[E,A,B]} + \im B) \cap A\cW^*_{[E,A,B]}, \\
%    &\subseteq (E\cV^*_{[E,A,B]} + \im B) \cap (A\cW^*_{[E,A,B]}+\im B), \\
    \end{aligned}
\end{equation}
\begin{equation}\label{eq:E(VcapW)_A(VcapW)}
\begin{aligned}
  E(\cV^*_{[E,A,B]} \cap \cW^*_{[E,A,B]}) &= E\cV^*_{[E,A,B]}\cap (A\cW^*_{[E,A,B]}+\im B),\\
  A(\cV^*_{[E,A,B]} \cap \cW^*_{[E,A,B]}) &= (E\cV^*_{[E,A,B]} + \im B) \cap A\cW^*_{[E,A,B]},
  \end{aligned}
\end{equation}
and
\begin{multline}\label{eq:EVBcapAWB}
  \hspace*{-3mm} E(\cV^*_{[E,A,B]}\cap\cW^*_{[E,A,B]})+\im B
  \ = \
  (E\cV^*_{[E,A,B]}+\im B)\cap (A\cW^*_{[E,A,B]}+\im B)\\
  = A(\cV^*_{[E,A,B]}\cap\cW^*_{[E,A,B]})+\im B.
\end{multline}
\item The augmented Wong limits $\cV^*_{[E,A,B]},\cW^*_{[E,A,B]}\subseteq\R^n$ are related to the Wong limits $\cV^*_{[[E,0],[A,B],0]}$, $\cW^*_{[[E,0],[A,B],0]}\subseteq\R^{n+m}$ of the augmented pencil $s[E,0]-[A,B]$ as follows:
\begin{equation}\label{eq:augWong_WongAug}
   \cV^*_{[E,A,B]} = [I_n,0]\cdot \cV^*_{[[E,0],[A,B],0]}\qquad\text{and}\qquad \cW^*_{[E,A,B]} = [I_n,0]\cdot \cW^*_{[[E,0],[A,B],0]}.
\end{equation}
\end{enumerate}
\end{Proposition}

\begin{Proof} ~\\[-4ex]
\begin{enumerate}[(a)]
 \item
The proof of this statement is straightforward
and hence omitted.

  \item
  By~\eqref{eq:AWS_nested} the two relations in~\eqref{eq:AWS_invariance} follow, which in turn immediately yield the subset inclusion ``$\subseteq$'' of
  the equations in~\eqref{eq:E(VcapW)_A(VcapW)}. To show the converse inclusions, let $z\in E\cV^*_{[E,A,B]}\cap (A\cW^*_{[E,A,B]}+\im B)$.
  Then there exist $v\in\cV^*_{[E,A,B]}$, $w\in\cW^*_{[E,A,B]}$, and~$u\in\R^m$ such that
  \[
      Ev = z = Aw + Bu.
  \]
  By~\eqref{eq:AWS_nested} we have $E\cW^*_{[E,A,B]} = (A\cW^*_{[E,A,B]}+\im B) \cap \im E$ and since $Aw+Bu\in(A\cW^*_{[E,A,B]}+\im B) \cap \im E$ there exists $\overline{w}\in\cW^*_{[E,A,B]}$ such that $E\overline{w} = Aw+Bu$ and hence $z=Ev=E\overline{w}$. Therefore, $v - \overline{w} \in \ker E \subseteq \cW^*_{[E,A,B]}$, which gives $v\in \cV^*_{[E,A,B]} \cap \cW^*_{[E,A,B]}$, thus $z = Ev \in E(\cV^*_{[E,A,B]}\cap\cW^*_{[E,A,B]})$. This shows $E\cV^*_{[E,A,B]}\cap (A\cW^*_{[E,A,B]}+\im B)\subseteq E(\cV^*_{[E,A,B]} \cap \cW^*_{[E,A,B]})$. The inclusion $(E\cV^*_{[E,A,B]} + \im B) \cap A\cW^*_{[E,A,B]}\subseteq A(\cV^*_{[E,A,B]} \cap \cW^*_{[E,A,B]})$ can be shown
  similarly and its proof is omitted.
%   analogously.
\\
  We show~\eqref{eq:EVBcapAWB}: For the first equality, observe that ``$\subseteq$'' follows from~\eqref{eq:AWS_invariance}. For ``$\supseteq$'' let $x\in(E\cV^*_{[E,A,B]}+\im B)\cap (A\cW^*_{[E,A,B]}+\im B)$, i.e., $x=Ev+b_1=Aw+b_2$ for some $v\in\cV^*_{[E,A,B]}, w\in\cW^*_{[E,A,B]}, b_1,b_2\in\im B$. Then
    \[
        v\in E^{-1} \{Aw+b_2-b_1\} \subseteq E^{-1} (A\cW^*_{[E,A,B]}+\im B) = \cW^*_{[E,A,B]}
    \]
    and hence $x\in E(\cV^*_{[E,A,B]}\cap\cW^*_{[E,A,B]})+\im B$. The second equality in~\eqref{eq:EVBcapAWB} can be proved similarly; we omit the proof.

    \item The proof of this statement can be easily inferred from the proof of~\cite[Lem.~2.1]{Berg19a}.\qedhere
\end{enumerate}
\end{Proof}

%============================================================================
\section{P-feedback forms}\label{Ssec:PFF}
%===============================================================

In this section, we recall the concept of P-feedback which allows a decoupling of the DAE~\eqref{eq:EAB}.
This has been successfully used for various purposes, cf.\ the survey~\cite{BergReis13a}. After that, we present the P-feedback form from~\cite{LoisOzca91}, which is a canonical form. As a new contribution, we
 derive  a quasi P-feedback form using the augmented Wong sequences.
Relevant system theoretic  information can be read off this form.
  Apart from allowing a calculation via the simple subspace sequences, this also provides some geometric insight in the decoupling.

Concerning applications, the new quasi P-feedback form may be advantageous for instance in observer design problems for differential-algebraic systems. The construction of regular and freely initializable observers in the proof of~\cite[Thm.~3.8]{BergReis17c} completely relies on the  P-feedback form. However, in order to implement this design procedure, a method with lower complexity would be favorable, for which the new quasi P-feedback form is predestined.

%______________________________________________________________________
\subsection{P-feedback equivalence}
We recall the notion of P-feedback equivalence for systems $[E,A,B]\in\Sigma_{\ell,n,m}$, see e.g.~\cite{BergReis13a}.
Here and in the
following, $GL_p(\R)$ denotes the set of all invertible matrices in~$\R^{p\times p}$,  $p\in\N$.

\begin{Definition}[P-feedback equivalence]\label{Def:PF-equiv}
Two systems
$[E_1, A_1, B_1], [E_2, A_2, B_2]  \in \Sigma_{\ell,n,m}$
are called
\emph{P-feedback equivalent}, if
\begin{equation}
\begin{aligned}
&\exists\, S\in \Gl_\ell(\R), T\in \Gl_n(\R), V\in \Gl_m(\R), F_P\in \R^{m\times n}: \\
&\begin{bmatrix} sE_1-A_1, & -B_1\end{bmatrix}
=
S
\begin{bmatrix} sE_2-A_2, & -B_2 \end{bmatrix}
\begin{bmatrix} T & 0 \\  F_P & V \end{bmatrix}\,;
\end{aligned}\label{eq:P-feedbequiv}
\end{equation}
we write
\[
 [E_1,A_1,B_1] \ {\cong_{P}} \  [E_2, A_2, B_2]
 \quad
    \text{or, if necessary,}\quad
     [E_1,A_1,B_1] \ \overset{S,T,V,F_P}{\cong_{P}} \  [E_2, A_2, B_2]\,.
\]
%\begin{align*}
%    & [E_1,A_1,B_1] \ {\cong_{P}} \  [E_2, A_2, B_2]\\
%    \text{or, if necessary}\quad & [E_1,A_1,B_1] \ \overset{S,T,V,F_P}{\cong_{P}} \  [E_2, A_2, B_2]\,.
%\end{align*}
\end{Definition}

\begin{Remark}
P-feedback equivalence is an equivalence relation on $\Sigma_{\ell,n,m}$:
\begin{itemize}
    \item Reflexivity: Clear with $S=I$, $T=I$, $V=I$, $F_P=0$.
    \item Symmetry: For $[E_1,A_1,B_1] \overset{S,T,V,F_P}{\cong_{P}}  [E_2, A_2, B_2]$ we have that
\[
    [E_2,A_2,B_2] \ \overset{S^{-1},T^{-1},V^{-1},-V^{-1}F_PT^{-1}}{\cong_{P}} \  [E_1, A_1, B_1],
\]
which can be verified by observing that
\[
    \begin{bmatrix} T & 0\\ F_P & V\end{bmatrix}^{-1} = \begin{bmatrix} T^{-1} & 0\\ -V^{-1}F_PT^{-1} & V^{-1}\end{bmatrix}.
\]
\item Transitivity: For $[E_1,A_1,B_1] \overset{S_1,T_1,V_1,F_1}{\cong_{P}}  [E_2, A_2, B_2]\overset{S_2,T_2,V_2,F_2}{\cong_{P}}  [E_3, A_3, B_3]$ we have
\[
    [E_1,A_1,B_1] \overset{S_1 S_2,T_2 T_1 ,V_2 V_1,\widetilde{F}}{\cong_{P}}  [E_3, A_3, B_3]\quad\text{ where }\quad \widetilde{F} = F_2 T_1 + V_2 F_1.
\]
\end{itemize}
\end{Remark}

The augmented Wong sequences change under P-feedback as shown in the following result.

\begin{Lemma}[Augmented Wong sequences under P-feedback]\label{lem:Wong-Pfb}
If the systems  $[E_1, A_1, B_1]$, $[E_2, A_2, B_2] \in \Sigma_{\ell,n,m}$
are P-feedback equivalent $[E_1, A_1, B_1] \overset{S,T,V,F_P}{\cong_{P}}  [E_2, A_2, B_2]$, then
\begin{equation*}\label{eq:rel-wong-P}
\begin{aligned}
    \forall\, i\in\N_0:\quad  \cV_{[E_1,A_1,B_1]}^i = T^{-1} \cV_{[E_2,A_2,B_2]}^i\quad \ \text{and} \quad
      \cW_{[E_1,A_1,B_1]}^i = T^{-1} \cW_{[E_2,A_2,B_2]}^i.
    \end{aligned}
\end{equation*}
\end{Lemma}
\begin{proof} We prove the first statement by induction. It is clear that $\cV_{[E_1,A_1,B_1]}^0 = T^{-1} \cV_{[E_2,A_2,B_2]}^0$.
Assume that $\cV_{[E_1,A_1,B_1]}^i = T^{-1} \cV_{[E_2,A_2,B_2]}^i$ for some $i\geq 0$.
Then~\eqref{eq:P-feedbequiv} yields
\begin{align*}
     \cV_{[E_1,A_1,B_1]}^{i+1}
    &= A_1^{-1} (E_1 \cV_{[E_1,A_1,B_1]}^i + \im  B_1)\\
    &= \setdef{x\in\R^n}{ \begin{array}{l} \exists\, y\in\cV_{[E_1,A_1,B_1]}^i\ \exists\, u\in\R^m:\\[1.5mm] (SA_2 T + SB_2F_P) x = SE_2T y + SB_2V u\end{array}}\\
    &= \setdef{x\in\R^n}{ \exists\, z\in\cV_{[E_2,A_2,B_2]}^i\ \exists\, v\in\R^m:\ A_2 T x = E_2 z + B_2 v}\\
    &= T^{-1} \left( A_2^{-1} (E_2 \cV_{[E_2,A_2,B_2]}^i + \im  B_2)\right) = T^{-1} \cV_{[E_2,A_2,B_2]}^{i+1}.
\end{align*}
The proof of the second statement is similar
% statement  on $\cW_{[E_1,A_1,B_1]}^i$ and $\cW_{[E_2,A_2,B_2]}^i$ is analogous
and omitted.
\end{proof}

%_____________________________________________________________________
\subsection{P-feedback form (PFF)}

For the definition of the P-feedback form we need to introduce some further notation. For $k\in\N$, consider the matrices
\[
N_k : =\left[\SmallNilBlock{0}{1}{2ex}\right]\in\R^{k\times k},
\quad
K_k   :=\left[\SmallRectBlock{0}{1}{2ex}\right],\ L_k :=\left[\SmallRectBlock{1}{0}{2ex}\right]\in\R^{(k-1)\times k},
\]
where $K_k = L_k = 0_{0\times 1}$ for $k=1$.
We set, for some multi-index $\bsalpha=(\alpha_1,\ldots,\alpha_k)\in\N^k$, \
 $|\bsalpha| = \alpha_1 + \ldots + \alpha_k$
 and introduce the notation
\begin{align*}
  N_\bsalpha &:= \diag(N_{\alpha_1},\ldots,N_{\alpha_k}) \in\R^{|\bsalpha|\times|\bsalpha|},\\
  K_\bsalpha &:= \diag(K_{\alpha_1},\ldots,K_{\alpha_k}) \in\R^{(|\bsalpha|-k)\times|\bsalpha|},\\
  L_\bsalpha &:= \diag(L_{\alpha_1},\ldots,L_{\alpha_k}) \in\R^{(|\bsalpha|-k)\times|\bsalpha|}.
\end{align*}
By $e_i^{[n]}$ we denote the~$i$-th unit vector in~$\R^n$
and define
% For a multi-index $\bsalpha=(\alpha_1,\ldots,\alpha_k)\in\N^k$ we then define
\[
    E_\bsalpha := \diag(e_{\alpha_1}^{[\alpha_1]},\ldots,e_{\alpha_k}^{[\alpha_k]})\in\R^{|\bsalpha|\times k},
    \qquad
 \text{for} \quad   \bsalpha=(\alpha_1,\ldots,\alpha_k)\in\N^k.
\]

\begin{Definition}[P-feedback form]\label{def:PFF}
The system $[E,A,B]\in\Sigma_{\ell,n,m}$
 is said to be in
\textit{P-feedback form}~(PFF), if
\begin{equation}
 \label{eq:Pform}
[E, A,B]
\ = \
\left[
  \begin{smallbmatrix}
   K_{\bsalpha} & 0 & 0 & 0 & 0 & 0\\[-0.5ex]
   \vphantom{N_{\bsalpha}^\top} 0 &  I_{|\bsbeta|} & 0 & 0& 0& 0\\
   0 & 0 & I_{n_{\overline{c}}} & 0 & 0 & 0\\
   \vphantom{I_{|\bskappa|}} 0 & 0 & 0 & N_\bsgamma & 0 & 0 \\[-0.5ex]
   0 & 0 & 0 & 0 & K_{\bsdelta}^\top & 0 \\[-0.5ex]
   0 & 0 & 0 & 0 & 0 & K_{\bskappa}^\top
 \end{smallbmatrix},
 \begin{smallbmatrix}
   L_{\bsalpha} & 0 & 0 & 0 & 0 & 0\\[-0.5ex]
   \vphantom{I_{|\bsalpha|}} 0 & N_{\bsbeta}^\top & 0 & 0& 0& 0 \\
   \vphantom{I_{n_{\overline{c}}}} 0 & 0 & A_{\overline{c}} & 0 & 0 & 0 \\
   0 & 0 & 0 & I_{|\bsgamma|} & 0 & 0\\[-0.5ex]
   0 & 0 & 0 & 0 & L_{\bsdelta}^\top & 0 \\[-0.5ex]
   0 & 0 & 0 & 0 & 0 & L_{\bskappa}^\top
  \end{smallbmatrix},
  \begin{smallbmatrix}
    \vphantom{K_{\bsbeta}}0&0&0 \\[-0.5ex]
    \vphantom{I_{|\bsalpha|}}\vphantom{N_{\bsalpha}^\top}E_{\bsbeta}&0&0  \\
    \vphantom{I_{n_{\overline{c}}}}0&0&0\\
    \vphantom{I_{|\bskappa|}}0 &0&0\\[-0.5ex]
    \vphantom{L_{\bsdelta}^\top} 0&0&0 \\[-0.5ex]
    \vphantom{K_{\bsgamma}^\top}0 & 0 & E_{\bskappa}
  \end{smallbmatrix}
  \right],
%\left[
%  \left[\begin{matrix} I_{|\bsalpha|} & 0 & 0 & 0& 0& 0\\ 0 & L_{\bsbeta} & 0 & 0 & 0 & 0\\ 0 & 0 & K_{\bsgamma}^\top & 0 & 0 & 0 \\ 0 & 0 & 0 & L_{\bsdelta}^\top & 0
%  & 0 \\ 0 & 0 & 0 & 0 & N_\bskappa & 0  \\ 0 & 0 & 0 & 0 & 0 & I_{n_{\overline{c}}}
% \end{matrix}\right],
%  \left[\begin{matrix} N_{\bsalpha}^\top & 0 & 0 & 0& 0& 0\\ 0 & K_{\bsbeta} & 0 & 0 & 0 & 0\\ 0 & 0 & L_{\bsgamma}^\top & 0 & 0 & 0\\ 0 & 0 & 0 & K_{\bsdelta}^\top & 0 & 0
%   \\ 0 & 0 & 0 & 0 & I_{|\bskappa|} & 0 \\ 0 & 0&0 & 0 & 0 & A_{\overline{c}}\end{matrix}\right],
%  \left[\begin{matrix} E_{\bsalpha}&0&0  \\ 0&0&0 \\ 0 &E_{\bsgamma}&0\\ 0&0&0\\ 0&0&0 \\ 0 &0&0\end{matrix}\right]
%\right],
\end{equation}
where $\bsalpha\in\N^{n_\bsalpha},\bsbeta\in\N^{n_\bsbeta},\bsgamma\in\N^{n_\bsgamma},\bsdelta\in\N^{n_\bsdelta},\bskappa\in\N^{n_\bskappa}$ are multi-indices and $A_{\overline{c}}\in\R^{n_{\overline{c}}\times n_{\overline{c}}}$.
\end{Definition}

We like to note that the~PFF of a system~$[E,A,B]$ can be viewed as a Kronecker canonical form~(KCF)
of the augmented pencil $s[E,0]-[A,B]$ with some additional structure, as shown in~\cite[Rem.~3.10]{BergReis13a}. This is remarkable
 because P-feedback equivalence induces an equivalence relation on $\R^{\ell\times(n+m)}[s]$ which is a subrelation of the system equivalence used to obtain the Kronecker canonical form, and hence it is not clear whether the Kronecker canonical form of~$s[E,0]-[A,B]$ is contained in each of the smaller equivalence classes.

We use the connection between the KCF and the PFF to show that two P-feedback equivalent systems have the same PFF up to permutation of the entries of $\bsalpha, \bsbeta, \bsgamma, \bsdelta, \bskappa$ and similarity of $A_{\overline{c}}$.

\begin{Proposition}[Uniqueness of indices for PFF]\label{Prop:Indices-uniqueness-PFF}
Let $[E_i,A_i,B_i]\in\Sigma_{\ell,n,m}$, $i=1,2$,
be in PFF~\eqref{eq:Pform} with corresponding multi-indices
$\bsalpha_i\in\N^{n_{\bsalpha_i}}, \bsbeta_i\in\N^{n_{\bsbeta_i}}, \bsgamma_i\in\N^{n_{\bsgamma_i}}, \bsdelta_i\in\N^{n_{\bsdelta_i}},\bskappa_i\in\N^{n_{\bskappa_i}}$, and  $A_{\overline{c},i}\in\R^{n_{\overline{c},i}\times n_{\overline{c},i}}$. If $[E_1, A_1, B_1] \cong_{P} [E_2, A_2, B_2]$, then
\[
    \bsalpha_1 =  P_\bsalpha \bsalpha_2,\quad \bsbeta_1 =  P_\bsbeta \bsbeta_2,\quad \bsgamma_1 =  P_\bsgamma \bsgamma_2,\quad \bsdelta_1 =   P_\bsdelta \bsdelta_2,\quad \bskappa_1 =  P_\bskappa \bskappa_2
\]
and
\[
    n_{\overline{c},1} = n_{\overline{c},2}, \quad A_{\overline{c},1} = H^{-1} A_{\overline{c},2} H,
\]
for permutation matrices $P_\bsalpha, P_\bsbeta, P_\bsgamma, P_\bsdelta, P_\bskappa$
of appropriate sizes and~$H\in\Gl_{n_{\overline{c},1}}(\R)$.
\end{Proposition}
\begin{proof}
  This result is a consequence of Lemma~\ref{lem:Wong-Pfb} and~\cite[Thm.~2.2~\&~Prop.~2.3]{BergReis15a}.
\end{proof}

We are now in the position to show that any $[E,A,B]\in\Sigma_{\ell,n,m}$ is P-feedback equivalent to a system in PFF.

\begin{Theorem}[PFF]\label{Thm:brundae}
For any system $[E,A,B]\in\Sigma_{\ell,n,m}$
  there exist $S\in \Gl_\ell(\R), T\in \Gl_n(\R), V\in \Gl_m(\R), F_P\in \R^{m\times n}$
such that
\[
    [SET,SAT+SBF_P,SBV]
\ \
\text{is in PFF~\eqref{eq:Pform}.}
\]
\end{Theorem}

The proof of Theorem 3.6 is omitted. It relies on
subtle transformations and is
proved in~\cite[Thm.~3.1]{LoisOzca91}~-- a paper coauthored by Nicos Karcanias.
%and is given as the proof of Theorem~3.1 in the work~\cite{LoisOzca91}

\begin{Example}[PFF]\label{Ex:P-form}
For an illustration of Theorem~\ref{Thm:brundae} we consider the system $[E,A,B]\in\Sigma_{7,6,3}$ with
\[
    E=\begin{smallbmatrix}-2 & -3 & 0 & -1 & -4 & -3\\ 1 & 4 & 3 & -1 & 4 & 4\\ 0 & -4 & -7 & 1 & -3 & -6\\ 0 & 2 & 1 & -2 & 2 & 1\\ 2 & 5 & 1 & -1 & 6 & 4\\ 2 & 4 & 2 & 1 & 5 & 5\\ -2 & 2 & 9 & -2 & 0 & 5
 \end{smallbmatrix},\quad
    A=\begin{smallbmatrix} -2 & -2 & 4 & 3 & 1 & -1\\ 0 & -3 & -3 & -2 & -5 & -3\\ -1 & 4 & 3 & 5 & 7 & 3\\ -1 & -2 & 3 & 1 & 0 & 0\\ 1 & 1 & 1 & -2 & 0 & 3\\ 4 & 0 & -5 & -5 & -2 & -2\\ 2 & -6 & 4 & -5 & -4 & -4
 \end{smallbmatrix},\quad
    B=\begin{smallbmatrix}1 & -1 & -1\\ 0 & 0 & 2\\ -1 & 2 & -3\\ 1 & -1 & 1\\ 0 & 0 & 2\\ 1 & -3 & 2\\ 5 & -9 & 3
\end{smallbmatrix}.
\]
With
\[\begin{aligned}
   S &= \begin{smallbmatrix} -15 & 2 & 4 & 5 & -6 & -6 & 4\\ -16 & -1 & 2 & 9 & -8 & -5 & 3\\ -3 & -1 & 0 & 3 & -2 & 0 & 0\\ 8 & 2 & 0 & -5 & 4 & 2 & -1\\ -1 & 0 & 0 & 1 & -1 & 0 & 0\\ -6 & 0 & 1 & 3 & -3 & -2 & 1\\ -4 & 0 & 1 & 2 & -2 & -1 & 1
 \end{smallbmatrix},&\quad
   T &= \begin{smallbmatrix} -17 & 10 & -13 & -3 & -8 & 6\\ 13 & -6 & 9 & 2 & 6 & -4\\ -7 & 4 & -5 & -1 & -3 & 2\\ 6 & -3 & 4 & 1 & 3 & -2\\ -5 & 2 & -3 & 0 & -2 & 1\\ 3 & -2 & 2 & 0 & 1 & -1
\end{smallbmatrix},\\
   V &= \begin{smallbmatrix}   2 & 0 & -1\\ 1 & 0 & -1\\ 0 & -1 & -1 \end{smallbmatrix},&
   F_P &= \begin{smallbmatrix} 3 & 3 & -2 & -2 & 0 & 2\\ -14 & 10 & -12 & -3 & -7 & 6\\ 7 & -4 & 6 & 3 & 4 & -3
\end{smallbmatrix},
\end{aligned}
\]
it can be verified that
\[\begin{aligned}
  SET &= \diag(K_1,I_2, I_1,N_1, K_2^\top, K_1^\top),\\
  S(AT+BF_P) &= \diag(L_1,N_2^\top,A_{\overline{c}},1,L_2^\top,L_1^\top),\\
  SBV &= \diag (0_{0\times 0},e_2^{[2]},0_{1\times 0},0_{1\times 0}, e_2^{[2]}, e_1^{[1]}),
  \end{aligned}
\]
where $A_{\overline{c}} = [1]$; therefore $[E,A,B]$ is P-feedback equivalent to a system in the PFF~\eqref{eq:Pform} with $\bsalpha=(1)$, $\bsbeta=(2)$, $n_{\overline{c}}=1$, $\bsgamma=(1)$, $n_\bsdelta=0$, $\bskappa=(2,1)$. The details on how to obtain this transformation as well as numerical considerations are out of the scope of this contribution. We will
however revisit this example in the next section in the context of the quasi
P-feedback forms and will briefly discuss numerical issues in Remark~\ref{rem:numerics}.
\end{Example}

%________________________________________________________________
\subsection{Quasi P-feedback form (QPFF)}

We will now weaken P-feedback forms to
\textit{quasi}  P-feedback forms.
Roughly speaking,
the latter is ``less canonical'' than the former,
it contains less zeros and ones. However~-- and this is
the important message~--
the relevant  system theoretic properties can be
read off the quasi P-feedback form and, moreover,
the form provides a geometric insight (as it is obtained via the augmented Wong sequences) and can be easily computed.

\begin{Definition}\label{def:QPFF-neu}
  The system $[E,A,B]\in\Sigma_{\ell,n,m}$ is said to be in \emph{quasi P-feedback form~(QPFF)}, if
  \begin{equation}\label{eq:QPFF}
     [E,A,B] = \left[\begin{bmatrix} E_{11} & E_{12} & E_{13} \\ 0 & E_{22} & E_{23} \\ 0 & 0 & E_{33} \end{bmatrix}, \begin{bmatrix} A_{11} & A_{12} & A_{13} \\ 0 & A_{22} & A_{23} \\ 0 & 0 & A_{33} \end{bmatrix}, \begin{bmatrix} B_{11} & 0 & B_{13} \\ 0 & 0 & 0 \\ 0 & 0 & B_{33} \end{bmatrix}\right],
  \end{equation}
  where
  \begin{enumerate}[(i)]
     \item \label{item:QPFF11} $[E_{11},A_{11},B_{11}]\in\Sigma_{\ell_1,n_{1},m_{1}}$ with $\ell_1 < n_1 + m_1$, $\rk E_{11} = \rk_\C [\lambda E_{11} - A_{11}, B_{11}] = \ell_1$ for all $\lambda\in\C$ and $\rk B_{11} = m_1$,
     \item \label{item:QPFF22}  $E_{22},A_{22}\in\R^{\ell_{2}\times n_{2}}$ with $\ell_2 = n_2$ and $E_{22}\in\Gl_{n_2}(\R)$,
     \item \label{item:QPFF33}$[E_{33},A_{33},B_{33}]\in\Sigma_{\ell_{3},n_{3},m_{3}}$ satisfies $\rk_\C [\lambda E_{33}-A_{33},B_{33}] = n_{3} + m_3$ for all $\lambda\in\C$
  \end{enumerate}
  and the remaining matrices have suitable sizes.

  Furthermore, a~QPFF~\eqref{eq:QPFF} with zero off-diagonal blocks (i.e.\ $E_{12}=A_{12}=0$, $E_{13}=A_{13} = 0$, $B_{13}=0$, $E_{23}=A_{23}=0$) is called \emph{decoupled~QPFF}.
\end{Definition}

\begin{Remark}
The three conditions in  Definition~\ref{def:QPFF-neu}
describe control theoretic properties as follows (see the survey~\cite{BergReis13a} for the different notions of controllability):
\begin{enumerate}[(i)]
   \item
  The system $[E_{11},A_{11},B_{11}]$ in the~QPFF~\eqref{eq:QPFF} is completely controllable; the input is not constrained and not redundant.
   \item The~ODE system $[E_{22},A_{22},0]$ is uncontrollable.
   \item The system $[E_{33},A_{33},B_{33}]$  has only
    the trivial solution; in particular, the system is trivially behaviorally controllable but the input corresponding to~$B_{33}$ is maximally constrained (because it has to be zero).
\end{enumerate}
\end{Remark}

The following result will be helpful in due course.

\begin{Proposition}\label{Prop:decoupledQPFF}
Any system $[E,A,B]\in\Sigma_{\ell,n,m}$ in~QPFF~\eqref{def:QPFF-neu}
is P-feedback equivalent to a system in \emph{decoupled}~QPFF with
 identical  diagonal blocks.
\end{Proposition}

\begin{proof}
 The proof is structurally similar to the proof of~\cite[Thm.~2.6]{BergTren12}, however, due to the presence of the input some technical adjustments are necessary. Two technical results required for the proof are collected in the Appendix~\ref{app:ProofDecoupledQPFF}.

To show that any~QPFF can be decoupled, we prove existence of matrices $G_S,H_S,F_S$ and $G^x_T, G_T^u, H_T^{x}, H_T^{u}, F^x_T$ of appropriate sizes such that
\begin{multline*}
   \begin{bmatrix} sE_{11}-A_{11} & sE_{12}-A_{12} & sE_{13}-A_{13} \\
  0 & sE_{22}-A_{22} & sE_{23}-A_{23} \\ 0 & 0 & sE_{33}-A_{33}\end{bmatrix} \begin{bmatrix} I & G_T^x & H_T^{x}\\0 & I & F_T^x \\ 0 & 0 & I \end{bmatrix} + \begin{bmatrix} B_{11} & 0 & B_{13} \\ 0 & 0 & 0 \\ 0 & 0 & B_{33} \end{bmatrix} \begin{bmatrix}0 & G_T^u & H_T^{u} \\ 0 & 0 & 0 \\ 0 & 0 & 0\end{bmatrix} \\
  = \begin{bmatrix}I & -G_S & -H_S \\ 0 & I & -F_S \\ 0 & 0 & I \end{bmatrix} \begin{bmatrix} sE_{11}-A_{11} & 0 & 0 \\
  0 & sE_{22}-A_{22} & 0 \\ 0 & 0 & sE_{33}-A_{33} \end{bmatrix}
\end{multline*}
and
\[
    \begin{bmatrix} B_{11} & 0 & B_{13} \\ 0 & 0 & 0 \\ 0 & 0 & B_{33} \end{bmatrix}  =  \begin{bmatrix}I & -G_S & -H_S \\ 0 & I & -F_S \\ 0 & 0 & I \end{bmatrix}\begin{bmatrix} B_{11} & 0 & 0 \\ 0 & 0 & 0 \\ 0 & 0 & B_{33} \end{bmatrix}.
\]
This holds
if, and only if, the following matrix equations have solutions:
\begin{subequations}\label{eq:GiFiHi}
\begin{align}
   &\begin{aligned}
     0 &= A_{12}  + [A_{11}, -B_{11}] \begin{smallbmatrix} G_T^x \\ G_T^u \end{smallbmatrix} + G_S A_{22}, \\[-0.5ex]
     0 &= E_{12} + [E_{11},0] \begin{smallbmatrix} G_T^x \\ G_T^u \end{smallbmatrix} + G_S E_{22};
   \end{aligned} \label{eq:Gi}\\[1ex]
  &\begin{aligned}
      0 &= [A_{23},0] + A_{22} [ F_T^x, 0] + F_S [A_{33},-B_{33}],\\[-0.5ex]
      0 &= [E_{23},0] + E_{22} [F_T^x, 0] + F_S [E_{33},0];
   \end{aligned}\label{eq:Fi}\\[1ex]
  &\begin{aligned}
      0 &= {A_{12} F_T^x + A_{13} + [A_{11},-B_{11}] \begin{smallbmatrix} H_T^{x} \\ H_T^{u} \end{smallbmatrix} + H_S A_{33}},\\[-0.5ex]
      0 &= {E_{12} F_T^x + E_{13} + [E_{11},0] \begin{smallbmatrix} H_T^{x} \\ H_T^{u} \end{smallbmatrix} + H_S E_{33}},\\[-0.5ex]
      0 & = - B_{13} - H_S B_{33}.
     \end{aligned}\label{eq:Hi}
\end{align}
\end{subequations}
In the following we show that each of the sets of equations above admits a solution, where we use Lemmas~\ref{lem:twoEqs2genSylvester} and~\ref{lem:Sylvester-equation}.
\\

\noindent
\emph{We show that~\eqref{eq:Gi} has a solution.}
\\
Clearly, \eqref{eq:Gi} has the form~\eqref{eq:twoMatrixEqs} with $\mathrm{A} = [A_{11},-B_{11}]$, $\mathrm{D} = A_{22}$, $\mathrm{C} = [E_{11},0]$, $\mathrm{B} = E_{22}$.
Since $E_{22}$ is invertible, there exists~$\lambda\in\R$ such that $\lambda E_{22}-A_{22}$ is also invertible and hence the assumption of Lemma~\ref{lem:twoEqs2genSylvester} is satisfied.
Therefore, it suffices to show solvability of the associated generalized Sylvester equation~\eqref{eq:genSylvester}. By assumption, $\rank[\lambda E_{11} - A_{11},B_{11}]=\ell_1$ for all $\lambda\in\C\cup\{\infty\}$, i.e.\ no rank drop occurs at all, so all assumptions of Lemma~\ref{lem:Sylvester-equation} are satisfied and existence of a solution $\begin{smallbmatrix} G_T^x \\ G_T^u \end{smallbmatrix}$ and $G_S$ of~\eqref{eq:Gi} is shown.
\\

\noindent
\emph{We show that~\eqref{eq:Fi} has a solution.}
\\
We consider first a relaxed version of~\eqref{eq:Fi} by replacing $[F_T^x,0]$ by $[F_T^x,F_T^u]$ in both equations of~\eqref{eq:Fi}. With $\mathrm{A} = A_{22}$, $\mathrm{D} = [A_{33},-B_{33}]$, $\mathrm{C}=E_{22}$, $\mathrm{B}=[E_{33},0]$ we again see that all assumptions of Lemmas~\ref{lem:twoEqs2genSylvester} and Lemma~\ref{lem:Sylvester-equation} are satisfied ensuring solvability of the relaxed version of~\eqref{eq:Fi}. From the second relaxed equation of~\eqref{eq:Fi} we see that, in particular,
\[
   0 = 0 + E_{22} F_T^u + 0,
\]
which, due to the invertibility of~$E_{22}$,
 immediately yields that~$F_T^u=0$.
 This shows solvability of the original equations~\eqref{eq:Fi}.
\\

\noindent
\emph{We show that~\eqref{eq:Hi} has a solution.}
\\
Since by assumption $[A_{33},-B_{33}]$ has full column rank, there exists an invertible row operation $R_{33}\in\R^{\ell_3\times \ell_3}$ such that
\[
   [A_{33}, -B_{33}] = R_{33} \begin{bmatrix} A_{33}^x & 0 \\ 0 & I_{m_3}\end{bmatrix},\quad A_{33}^x\in\R^{(\ell_3-m_3)\times n_3}.
\]
Define $[H_S^x,H_S^u]:=H_S R_{33}$, then the last equation of~\eqref{eq:Hi} simplifies to
\[
   0 = - B_{13} + H_S^u.
\]
This is solvable with $H_S^u = B_{13}$. Let $\begin{smallbmatrix} E^x_{33} \\ E^u_{33} \end{smallbmatrix} :=R_{33}^{-1} E_{33}$ with $E_{33}^x\in\R^{(\ell_3-m_3)\times n_3}$ and $E_{33}^u\in\R^{m_3\times n_3}$, then the first two equations of~\eqref{eq:Hi} have the form
\begin{equation}\label{eq:HTHS}
   \begin{aligned}
      0 &= \widetilde{A}_{13} + [A_{11},-B_{11}] \begin{smallbmatrix} H_T^{x} \\ H_T^{u} \end{smallbmatrix} + H_S^x A^x_{33},\\[-0.5ex]
      0 &= \widetilde{E}_{13} + [E_{11},0] \begin{smallbmatrix} H_T^{x} \\ H_T^{u} \end{smallbmatrix} + H_S^x E^x_{33},
   \end{aligned}
\end{equation}
where $\widetilde{A}_{13} = A_{12} F_T^x + A_{13}$ and $\widetilde{E}_{13} = E_{12}F_T^x + E_{13} + H_S^u E^u_{33}$. Since $A^x_{33}$ has full column rank, the pencil $sE^x_{33}-A^x_{33}$ has full polynomial column rank and by assumption  $\rank(\lambda [E_{11},0] - [A_{11},-B_{11}])=\ell_1$ for all $\lambda\in\C\cup\{\infty\}$. So Lemma~\ref{lem:Sylvester-equation} together with Remark~\ref{rem:transposeTwoMatrixEqs} is applicable to~\eqref{eq:HTHS} and guarantees existence of $H_T^{x}$, $H_T^{u}$, $H_S^x$ satisfying~\eqref{eq:HTHS}.
Now set $H_S = R_{33}^{-1} [H_S^x,H_S^u]$,
and the proof is complete.
\end{proof}

A notable observation from the proof of Proposition~\ref{Prop:decoupledQPFF} is that no input transformation is needed (i.e., $V=I$ in~\eqref{eq:P-feedbequiv}) to arrive at a decoupled~QPFF.

We  stress some important properties of the augmented Wong limits for systems in
decoupled~QPFF.

\begin{Lemma}\label{lem:Wong-decoupledQPFF}
For any $[E,A,B]\in\Sigma_{\ell,n,m}$ in decoupled~QPFF, the augmented Wong sequences satisfy:
\[
   \cV^*_{[E,A,B]}\cap\cW^*_{[E,A,B]} = \R^{n_1}\times\{0\}^{n_2+n_3},\quad \cV^*_{[E,A,B]} = \R^{n_1+n_2} \times \{0\}^{n_3}
\]
and
\[
   E(\cV^*_{[E,A,B]} \cap \cW^*_{[E,A,B]}) = \R^{\ell_1}\times \{0\}^{\ell_2+\ell_3},\quad E\cV^*_{[E,A,B]} = \R^{\ell_1+\ell_2}\times\{0\}^{\ell_3}.
\]
Furthermore, we have that $m_1 = m - m_2 -m_3$, where $m_2 = \dim \ker B$ and
\begin{equation}\label{eq:QPFF-m3}
    m_3 = \dim \Big( \im B \cap (\{0\}^{\ell_1+\ell_2}\times \R^{\ell_3})\Big).
\end{equation}
In particular (in view of Lemma~\ref{lem:Wong-Pfb} and Proposition~\ref{Prop:decoupledQPFF}), two~QPFFs which are P-feedback equivalent have the same block sizes in $E$, $A$ and $B$.
\end{Lemma}

\begin{Proof}
\emph{Step 1}: We show $\cV^*_{[E,A,B]} = \R^{n_1+n_2} \times \{0\}^{n_3}$.
\\
Since $[E,A,B]$ is in decoupled~QPFF, $\cV^*_{[E,A,B]} = \cV^*_{[E_{11},A_{11},B_{11}]}\times \cV^*_{[E_{22},A_{22},0]}\times \cV^*_{[E_{33},A_{33},B_{3}]}$. Since both~$E_{11}$ and~$E_{22}$ have full row rank, they are surjective. It  follows inductively that
$\cV^i_{[E_{11},A_{11},B_{11}]}=\R^{n_1}$ and $\cV^i_{[E_{22},A_{22},0]}=\R^{n_2}$ for all~$i\geq 0$. It remains to show that $\cV^*_{[E_{33},A_{33},B_{33}]}=\{0\}^{n_3}$.
Since $\lambda [E_{33},0]-[A_{33},B_{33}]$ has full column rank for all $\lambda\in\C$, its quasi-Kronecker form~\cite{BergTren13}
has  only a nilpotent and a overdetermined part, in particular, $\cV^*_{[[E_{33},0],[A_{33},B_{33}],0]} = \{0\}^{n_3+m_3}$. Now the claim follows from~\eqref{eq:augWong_WongAug}.
\\

\noindent
\emph{Step 2}: We show $\R^{n_1}\times\{0\}^{n_2}\times\{0\}^{n_3} \subseteq \cW^*_{[E,A,B]} \subseteq \R^{n_1}\times\{0\}^{n_2}\times\R^{n_3}$.

Again, since $[E,A,B]$ is in decoupled~QPFF we have
$\cW^*_{[E,A,B]} = \cW^*_{[E_{11},A_{11},B_{11}]}\times \cW^*_{[E_{22},A_{22},0]}\times \cW^*_{[E_{33},A_{33},B_{3}]}$. Thus, it suffices to show that $\cW^*_{[E_{11},A_{11},B_{11}]} = \R^{n_1}$ and $\cW^*_{[E_{22},A_{22},0]} = \{0\}$. The latter is a simple consequence of invertibility of~$E_{22}$ and that $\cW^0_{[E_{22},A_{22},0]} = \{0\}$; for the former we observe that the augmented matrix pencil $s[E_{11},0]-[A_{11},B_{11}]$ is an underdetermined~DAE in the sense of \cite{BergTren12} and hence $\cW^*_{[[E_{11},0],[A_{11},B_{11}],0]}=\R^{n_1+m_1}$. Invoking again~\eqref{eq:augWong_WongAug} we can conclude the claim.
\\

\noindent
\emph{Step 3}: We conclude the claimed properties of the augmented Wong-sequences in the statement of the lemma. The first two equations follow from Steps~1 and~2.
The third and fourth equation follow from the block structure of~$E$ and full row rank of~$E_{11}$ and $E_{22}$.
\\

\noindent
\emph{Step 4}: We show~\eqref{eq:QPFF-m3}. It is easy to see that
\[
    \im B \cap (\{0\}^{\ell_1+\ell_2}\times \R^{\ell_3}) = \{0\}^{\ell_1+\ell_2}\times \im B_{33},
\]
and  the full column rank of~$B_{33}$ yields $\dim \im B_{33} = m_3$. This proves~\eqref{eq:QPFF-m3}.
\end{Proof}

In the following we derive the~QPFF by choosing basis matrices according to the augmented Wong sequences. This has the advantage that the transformation provides some geometric insight. We are now in the position to show that any system $[E,A,B]$ is equivalent to a system in~QPFF.

\begin{Theorem}[Quasi P-feedback form]\label{thm:QPFF}
 Consider $[E,A,B]\in\Sigma_{\ell,n,m}$ with corresponding augmented Wong limits~$\cV^*_{[E,A,B]}$ and~$\cW^*_{[E,A,B]}$. Choose full column rank matrices $U_T\in\R^{n\times n_{1}}$, $R_T\in\R^{n\times n_{2}}$, $O_T\in\R^{n\times n_{3}}$, $U_S\in\R^{\ell\times \ell_{1}}$, $R_S\in\R^{\ell\times \ell_{2}}$, $O_S\in\R^{\ell\times \ell_{3}}$ such that
 \[\begin{aligned}
    \im U_T &= \cV^*_{[E,A,B]}\cap\cW^*_{[E,A,B]},\\
     \im R_T \oplus \im U_T  &= \cV^*_{[E,A,B]}, \\
     \im O_T \oplus  \im R_T \oplus\im U_T &= \R^n, \\[2ex]
    \im U_S &= E\cV^*_{[E,A,B]} \cap (A\cW^*_{[E,A,B]}+\im B),\\ % = E(\cV^*_{[E,A,B]}\cap\cW^*_{[E,A,B]}),\\
     \im R_S \oplus  \im U_S &= E\cV^*_{[E,A,B]},\\
    \im O_S \oplus  \im R_S \oplus  \im U_S &= \R^{\ell}
    \end{aligned}
 \]
 and, additionally,
 \begin{equation}\label{eq:cond-B-O2}
    \im B \subseteq \im [U_S, O_S].
 \end{equation}
 Let $T:=[U_T,R_T,O_T]$, $S:=[U_S,R_S,O_S]^{-1}$ and further choose (not necessarily full rank) matrices $F_{1}\in\R^{m\times n_{1}}$, $F_{2}\in\R^{m\times n_{2}}$ such that
% \[\begin{aligned}
%    A U_T &= U_S A_{11} - B F_{1}, &\quad
%    A R_T &= [U_S,R_S] \begin{smallbmatrix} A_{12} \\ A_{22} \end{smallbmatrix} - B F_{2}.\\
%    \end{aligned}
% \]
  \begin{equation}\label{eq:F1_F2_choice}
%  \begin{aligned}
      [0,I_{\ell_2+\ell_3}] S (A U_T + B F_1) = 0
\qquad \text{and}\qquad
      [0,0,I_{\ell_3}] S (A R_T + B F_2) = 0,
    % \end{aligned}
  \end{equation}
  and let $F_P = [F_{1},F_{2},0]$. Finally, let $V=[V_{1},V_2,V_3]$, where $V_{1},V_2,V_3$ are full column rank matrices with $\im V_1 \oplus \im V_2 \oplus \im V_3 = \R^m$ such that
 \begin{equation}\label{eq:V1_V2_prop}
     \im V_2 = \ker B,\quad\text{and}\quad \im [V_{1},V_2] = \ker [0,0,I_{\ell_3}] S B.
 \end{equation}
% Finally, let $V=[V_{1},V_3]$, where $V_{1},V_{3}$ are full column rank matrices such that
% \[
%     \im V_{1} = \ker [0,0,I_{\ell_3}] S B
%     \qquad \text{and}\qquad
%      \im V_{1} \oplus \im V_{3} = \R^m.
% \]
 Then $[SET,S(AT+BF_P),SBV]$ is in~QPFF~\eqref{eq:QPFF}.
\end{Theorem}

\begin{Proof}
\emph{Step 1}: We show that the block structure of the~QPFF~\eqref{eq:QPFF} is achieved. The subspace inclusions~\eqref{eq:AWS_invariance} imply that
\[
  \begin{aligned}
     \im E U_T &\subseteq \im U_S,&\quad \im A U_T &\subseteq \im U_S + \im B,\\
     \im E R_T &\subseteq \im [U_S,R_S],&\quad \im A R_T &\subseteq \im [U_S,R_S] + \im B,\\
     \im E O_T &\subseteq \im [U_S,R_S,O_S]=\R^n,&\quad \im A O_T &\subseteq \im[U_S,R_S,O_S]=\R^\ell,
  \end{aligned}
\]
hence there exists matrices $E_{11}, E_{12}, E_{13}, E_{22}, E_{23}, E_{33}, A_{11}, A_{12}, A_{13}, A_{22}, A_{23}, A_{33}, F_{1}, F_{2}$ such that
\begin{equation}\label{eq:QPFF-submatrices}
\begin{aligned}
   E U_T &= U_S E_{11}, &\quad A U_T &= U_S A_{11} - B F_{1},\\
   E R_T &= U_S E_{12} + R_S E_{22},& A R_T &= U_S A_{12} + R_S A_{22} - B F_{2},\\
   E O_T &= U_S E_{13} + R_S E_{23} + O_S E_{33}, &    A O_T &= U_S A_{13} + R_S A_{23} + O_S A_{33}.
   \end{aligned}
\end{equation}
Note that, in particular, $F_1$ and $F_2$ satisfy the equations~\eqref{eq:F1_F2_choice}, which hence have solutions. Conversely, for any solution $F_1$ and $F_2$ of~\eqref{eq:F1_F2_choice}, the matrices $A_{11}:=[I_{\ell_1},0,0] S (A U_T + B F_1)$ and $\begin{smallbmatrix} A_{12} \\ A_{22} \end{smallbmatrix} := [I_{\ell_1+\ell_2},0] S (A R_T + B F_2)$ are suitable choices for satisfying~\eqref{eq:QPFF-submatrices}.

Observe that~\eqref{eq:QPFF-submatrices} implies that $S E T$ and $S(AT+BF)$ have the desired block structure of a~QPFF~\eqref{eq:QPFF}.
Furthermore, since
\[
    \im U_S \cap \im B = E\cV^*_{[E,A,B]} \cap \im B  = \im[U_S,R_S] \cap \im B,
\]
one may always choose some~$O_S$
such  that~\eqref{eq:cond-B-O2}  holds.
Therefore, we may choose matrices $\widetilde{B}_1$, $\widetilde{B}_2$ such that
\[
    B = U_S \widetilde{B}_{1} + O_S\widetilde{B}_{2}
\]
holds, or, equivalently,
\[
    S B = \begin{bmatrix} \widetilde{B}_{1} \\ 0 \\ \widetilde{B}_{2} \end{bmatrix}.
\]
Finally, by construction we have $[0,0,I_{\ell_3}] S B V_{1} = 0$ and $SBV_2 = 0$, hence
\[
   S B V = \begin{bmatrix} B_{11} & 0 & B_{13} \\ 0 & 0 & 0 \\ 0 & 0 & B_{33} \end{bmatrix},
\]
which concludes Step~1. For later use we note that $0 = B_{33} x = [0,0,I_{\ell_3}] S B V_3 x$ implies that $V_3 x \in \ker [0,0,I_{\ell_3}] S B \cap \im V_3 = \im [V_1,V_2] \cap \im V_3 = \{0\}$, thus $B_{33}$ has full column rank. Similarly, $0=B_{11}x = [I_{\ell_1},0,0] S B V_1 x$ implies that $V_1 x \in \ker [I_{\ell_1},0,0] S B \cap \im V_1 \subseteq \ker [I_{\ell_1},0,0] S B \cap \ker[0,0,I_{\ell_3}] SB = \ker SB = \im V_2$, hence $V_1 x\in\im V_1 \cap \im V_2 = \{0\}$ which shows that $B_{11}$ is also of full column rank.
%For later use we note that $0 = B_{33} v = [0,0,I_{\ell_3}] S B V_3 v$ implies that $V_3 v \in \ker [0,0,I_{\ell_3}] S B \cap \im V_2 = \im V_1 \cap \im V_2 = \{0\}$, thus~$B_{33}$ has full column rank.
\\

\noindent\emph{Step 2}: We show that $[E_{11},A_{11},B_{11}]$ satisfies Definition~\ref{def:QPFF-neu}\,(i).
\\
\noindent\emph{Step 2a}: We show that $\rk E_{11} = \ell_1$.
Observe that $U_S E_{11} = E U_T$ yields
\[
    \im U_S E_{11} = \im E U_T = E(\cV^*_{[E,A,B]} \cap \cW^*_{[E,A,B]}) \overset{\rm Prop.~\ref{Prop:AWS_properties}}{=} E\cV^*_{[E,A,B]} \cap (A\cW^*_{[E,A,B]}+\im B) = \im U_S.
\]
As a consequence, the full column rank of~$U_S$ gives that~$E_{11}$ has full row rank.\\
\noindent\emph{Step 2b}: We show that $\cV^*_{[E_{11}, A_{11}, B_{11}]}\cap \cW^*_{[E_{11}, A_{11}, B_{11}]} = \R^{n_1}$.
\\
Set $[\widehat{E},\widehat{A},\widehat{B}] := [SET, S(AT+BF_P), SBV]$. Invoking Lemma~\ref{lem:Wong-Pfb} we can conclude that
\[
  \cV^*_{[\widehat{E},\widehat{A},\widehat{B}]}\cap\cW^*_{[\widehat{E},\widehat{A},\widehat{B}]}
  = T^{-1} (\cV^*_{[E,A,B]} \cap \cW^*_{[E,A,B]}) = T^{-1} (\im U_T) = \R^{n_1} \times \{0\}^{n_2+n_3}.
\]
From~\eqref{eq:EVBcapAWB} we may further infer that
\[\begin{aligned}
   \widehat{A}(\cV^*_{[\widehat{E},\widehat{A},\widehat{B}]}\cap\cW^*_{[\widehat{E},\widehat{A},\widehat{B}]}) &\subseteq \widehat{E}(\cV^*_{[\widehat{E},\widehat{A},\widehat{B}]}\cap\cW^*_{[\widehat{E},\widehat{A},\widehat{B}]}) + \im \widehat{B}.%, \\
%   \widehat{E}(\cV^*_{[\widehat{E},\widehat{A},\widehat{B}]}\cap\cW^*_{[\widehat{E},\widehat{A},\widehat{B}]}) &\subseteq \widehat{A}(\cV^*_{[\widehat{E},\widehat{A},\widehat{B}]}\cap\cW^*_{[\widehat{E},\widehat{A},\widehat{B}]}) + \im \widehat{B}.
   \end{aligned}
\]
Hence for all $x_1\in\R^{n_1}$, there exist $y_1, z_1\in\R^{n_1}$ and $u_1,v_1\in\R^{m_1}$, $u_2,v_2\in\R^{m_2}$, $u_3,v_3\in\R^{m_3}$ such that
\[
    \widehat{A} \begin{pmatrix} x_1 \\ 0 \\ 0 \end{pmatrix} = \widehat{E} \begin{pmatrix} y_1 \\ 0 \\ 0 \end{pmatrix} + \widehat{B} \begin{pmatrix} u_1 \\ u_2 \\ u_3 \end{pmatrix}\quad \text{ and }\quad {\widehat{A}} \begin{pmatrix} y_1 \\ 0 \\ 0 \end{pmatrix} = {\widehat{E}} \begin{pmatrix} z_1 \\ 0 \\ 0 \end{pmatrix} + \widehat{B} \begin{pmatrix} v_1 \\ v_2 \\ v_3 \end{pmatrix}.
\]
From the block diagonal structure, we can conclude that
\[\begin{aligned}
   A_{11} x_1 &= E_{11} y_1 + B_{11} u_1 + B_{13} u_3,\\
   0 &= B_{33} u_3.
   \end{aligned}
\]
Since $\ker B_{33} = \{0\}$, we find that $u_3 = 0$, thus we may conclude that
\[
   x_1 \in A_{11}^{-1} (E_{11} \{y_1\} + \im B_{11})\subseteq \cV^1_{[E_{11},A_{11},B_{11}]}.
\]
With the same  reasoning, one  may show that
\[
   y_1 \in A_{11}^{-1} (E_{11} \{z_1\} + \im B_{11})\subseteq \cV^1_{[E_{11},A_{11},B_{11}]}
\]
and hence $x_1\in\cV^2_{[E_{11},A_{11},B_{11}]}$.
Continuing this reasoning, it finally follows that $x_1\in\cV^k_{[E_{11},A_{11},B_{11}]}$ for all~$k\in\N$,
 and, since~$x_1$ is arbitrary, we have shown that $\cV^*_{[E_{11},A_{11},B_{11}]} = \R^{n_1}$.

Now, let $j^*,r^*\in\N$ be such that $\cW^*_{[\widehat{E},\widehat{A},\widehat{B}]}  = \cW^{j^*}_{[\widehat{E},\widehat{A},\widehat{B}]}$ and $\cW^*_{[E_{11}, A_{11}, B_{11}]} = \cW^{r^*}_{[E_{11}, A_{11}, B_{11}]}$ and set $q^* := \max\{j^*,r^*\}$. Again, take arbitrary $x = (x_1^\top, 0, 0)^\top \in \cV^*_{[\widehat{E},\widehat{A},\widehat{B}]} \cap \cW^*_{[\widehat{E},\widehat{A},\widehat{B}]}$, then $x\in \cW^{q^*}_{[\widehat{E},\widehat{A},\widehat{B}]}$ and Definition~\ref{Def:AWS} give that there exist $y_k\in\cW^{k}_{[\widehat{E},\widehat{A},\widehat{B}]}$ and $u_k\in\R^m$, $k=0,\ldots,q^*-1$,
 such that $\widehat Ex = \widehat A y_{q^*-1} + \widehat B u_{q^*-1}$ and $\widehat Ey_k = \widehat Ay_{k-1} + \widehat B u_{k-1}$ for all $k=0,\ldots,q^*-1$. Now we find that
\[
      \widehat Ay_{q^*-1}
      = \widehat E x - \widehat Bu_{q^*-1} \in \widehat E\big( \cV^*_{[\widehat{E},\widehat{A},\widehat{B}]} \cap \cW^*_{[\widehat{E},\widehat{A},\widehat{B}]}\big) + \im \widehat B
    \ \ \overset{\eqref{eq:EVBcapAWB}}{=}
    \ \
    \widehat A\big( \cV^*_{[\widehat{E},\widehat{A},\widehat{B}]} \cap \cW^*_{[\widehat{E},\widehat{A},\widehat{B}]}\big) + \im \widehat B,
\]
whence
\begin{align*}
    y_{q^*-1} &\in \widehat A^{-1}\Big(\widehat A\big( \cV^*_{[\widehat{E},\widehat{A},\widehat{B}]} \cap \cW^*_{[\widehat{E},\widehat{A},\widehat{B}]}\big) + \im \widehat B\Big)= \cV^*_{[\widehat{E},\widehat{A},\widehat{B}]} \cap \cW^*_{[\widehat{E},\widehat{A},\widehat{B}]} + \widehat A^{-1}(\im \widehat B)
   {{} \subseteq \cV^*_{[\widehat{E},\widehat{A},\widehat{B}]}}.
\end{align*}
Therefore,
\[
  y_{q^*-1} \in \cV^*_{[\widehat{E},\widehat{A},\widehat{B}]} \cap \cW^{q^*-1}_{[\widehat{E},\widehat{A},\widehat{B}]}.
\]
With the same  reasoning,
we may now conclude inductively
\[
  y_{k} \in \cV^*_{[\widehat{E},\widehat{A},\widehat{B}]} \cap \cW^{k}_{[\widehat{E},\widehat{A},\widehat{B}]},\quad k=q^*-2,\ldots,0.
\]
This implies that $y_k = (y_{k,1}^\top, 0, 0)^\top$ and $u_k = (u_{k,1}^\top, u_{k,2}^\top,u_{k,3}^\top)^\top$ for some $y_{k,1}\in\R^{n_1}$, $u_{k,1}\in\R^{m_1}$, $u_{k,2}\in\R^{m_2}$, $u_{k,3}\in\R^{m_3}$, $k=0,\ldots,q^*-1$, and hence
\begin{multline*}
    E_{11} y_{1,1} = B_{11} u_{0,1} + B_{13} u_{0,3} ,\ E_{11} y_{2,1} = A_{11} y_{1,1} + B_{11} u_{1,1} + B_{13} u_{1,3},\\
     \ldots,\ E_{11} y_{q^*-1,1} = A_{11} y_{q^*-2,1} + B_{11} u_{q^*-2,1} + B_{13} u_{q^*-2,3},\\
      E_{11} x_1 = A_{11} y_{q^*-1,1} + B_{11} u_{q^*-1,1} + B_{13} u_{q^*-1,3},
\end{multline*}
and
\[
    0 = B_{33}  u_{0,3} ,\ 0 = B_{33} u_{1,3},\ \ldots,\ 0 = B_{33} u_{q^*-2,3},\ 0 = B_{33} u_{q^*-1,3}.
\]
Therefore, we obtain $u_{k,3} = 0$ for all $k=0,\ldots,q^*-1$ and, as a consequence, $x_1\in\cW^*_{[E_{11}, A_{11}, B_{11}]}$. Since~$x_1$ was arbitrary we have proved that
\[
    \cV^*_{[E_{11}, A_{11}, B_{11}]}\cap \cW^*_{[E_{11}, A_{11}, B_{11}]} = \R^{n_1}.
\]

\noindent\emph{Step 2c}: We show that $\cV^*_{[[E_{11},0],[A_{11},B_{11}],0]}\cap \cW^*_{[[E_{11},0],[A_{11},B_{11}],0]} = \R^{n_1+m_1}$.
\\
It follows from Proposition~\ref{Prop:AWS_properties} and Step~2b that
\[
   [I_n,0] \cV^*_{[[E_{11},0],[A_{11},B_{11}],0]}\cap [I_n,0]\cW^*_{[[E_{11},0],[A_{11},B_{11}],0]} = \R^{n_1}.
\]
As shown in the proof of Proposition~\ref{Prop:AWS_properties} we have that
\[
    \cW^*_{[[E_{11},0],[A_{11},B_{11}],0]} = [I_{n_1},0]\cW^*_{[[E_{11},0],[A_{11},B_{11}],0]} \times \R^{m_1} = \R^{n_1}\times\R^{m_1}.
\]
Furthermore,
\[
\begin{array}{rcl}
  \cV^*_{[[E_{11},0],[A_{11},B_{11}],0]} &=&
   [A_{11}, B_{11}]^{-1}\big( [E_{11}, 0] \cV^*_{[[E_{11},0],[A_{11},B_{11}],0]} \big)= [A_{11}, B_{11}]^{-1}\big( \im E_{11} \big)\\
  &\overset{\rm Step~2a}{=} &
  [A_{11}, B_{11}]^{-1}\big( \R^{\ell_1} \big) = \R^{n_1 + m_1},
\end{array}
\]
which proves the claim.

\noindent{\emph{Step 2d: Conclusion of Step 2.}}
The result of Step 2c implies that the augmented matrix pencil $s[E_{11},0] - [A_{11},B_{11}]$ is in quasi-Kronecker form, see \cite{BergTren12}, and consists only
 of an underdetermined block. In particular,
$\ell_1<n_1+m_1$ and $\rk(\lambda [E_{11},0] - [A_{11},B_{11}]) = \ell_1$ for all $\lambda\in\C$. Full column rank of $B_{11}$ was already shown in Step 1.
\\

\noindent\emph{Step 3}: We show that $E_{22}$ is square and invertible.\\
\noindent\emph{Step 3a}: We show that $\im U_S \oplus \im E R_T = E \cV^*_{[E,A,B]}$.
\\
Since $\im R_T\subseteq\cV^*_{[E,A,B]}$ and $\im U_S\subseteq E\cV^*_{[E,A,B]}$ it follows that $\im U_S + \im E R_T \subseteq E \cV^*_{[E,A,B]}$. Furthermore,
\[
   E \cV^*_{[E,A,B]} = E \big((\cV^*_{[E,A,B]}\cap\cW^*_{[E,A,B]}) \oplus \im R_T\big)
   \subseteq E (\cV^*_{[E,A,B]}\cap\cW^*_{[E,A,B]}) + \im E R_T
   \overset{\eqref{eq:E(VcapW)_A(VcapW)}}{=} \im U_S + \im E R_T,
\]
hence it remains to be shown that the intersection of $\im U_S$ and $\im E R_T$ is trivial. Towards this goal, let $x\in\im U_S \cap \im E R_T$, then there exists $y\in\im R_T$ with $x=E y$ and, in view of~\eqref{eq:E(VcapW)_A(VcapW)}, there exists $z\in\cV^*_{[E,A,B]}\cap\cW^*_{[E,A,B]}$ such that $x = Ez$. Hence $z-y\in\ker E \subseteq \cW^*_{[E,A,B]}$. From $z\in\cW^*_{[E,A,B]}$ it then follows that $y\in\cW^*_{[E,A,B]}$ and, therefore, $y\in\cW^*_{[E,A,B]}\cap\im R_T = \{0\}$. This implies $x=0$ and completes  the proof of Step~3a.\\
\noindent\emph{Step 3b}: We show $\ell_2 = n_2$.
\\
From Step 3a we have $\ell_2 = \rk E R_T \leq n_2$ and hence it suffices to show that~$E R_T$ has full column rank. Let $v\in\R^{n_2}$ be such that $E R_T v = 0$, then $R_T v \in \im R_T \cap \ker E \subseteq \im R_T \cap \cW^*_{[E,A,B]} = \{0\}$ and due to full column rank of~$R_T$ the claim follows.\\
\noindent\emph{Step 3c}: We show full column rank of $E_{22}$.
\\
Let $v\in\R^{n_2}$ be such that $E_{22}v = 0$. Then by~\eqref{eq:QPFF-submatrices} we have $E R_T v = U_S E_{12} v$ and hence, invoking Step 3a, $E R_T v \in \im E R_T \cap \im U_S = \{0\}$. As already shown in Step 3b, $E R_T v =0$ implies $v=0$ and full column rank of $E_{22}$ is shown.
\\

\noindent\emph{Step 4}: We show that $[E_{33},A_{33},B_{33}]$ satisfies Definition~\ref{def:QPFF-neu}\,(iii).
% similar proof technique as in Step 4, Proof of Theorem~3.3 of [BergTren14]
\\
Assume there exist $\lambda\in\C$, $x_3\in\C^{n_3}$, $u_3\in\C^{m_3}$ such that $(\lambda E_{33} - A_{33}) x_3 + B_{33} u_3 = 0$. Then we have, according to~\eqref{eq:QPFF-submatrices}, that
\[
   (\lambda E - A) O_T x_3 = U_S (\lambda E_{13}-A_{13}) x_3 + R_S (\lambda E_{23}-A_{23}) x_3 - O_S B_{33} u_3.
\]
Writing the complex variables in terms of their real and imaginary parts, i.e., $\lambda = \mu + \imath \nu$, $x_3 = \overline{x}_3 + \imath \widehat{x}_3$ and $u_3 = \overline{u}_3 + \imath \widehat{u}_3$, we can conclude that
\[\begin{aligned}
    (\mu E - A)O_T \overline{x}_3 - \nu E O_T \widehat{x}_3 + O_S B_{33} \overline{u}_3 &\in \im [U_S,R_S],\\
    (\mu E - A)O_T \widehat{x}_3 + \nu E O_T \overline{x}_3 + O_S B_{33} \widehat{u}_3 &\in \im [U_S,R_S].\\
 \end{aligned}
\]
Furthermore,
\[
    \im(O_S B_{33} + U_S B_{13}) = \im [U_S,R_S,O_S] \begin{smallbmatrix} B_{13}\\ 0 \\ B_{33} \end{smallbmatrix}  = \im B V_2 \subseteq \im B,
\]
and hence
\[\begin{aligned}
    (\mu E - A)O_T \overline{x}_3 - \nu E O_T \widehat{x}_3 &\in \im[U_S,R_S]+\im B = E\cV^*_{[E,A,B]} + \im B,\\
    (\mu E - A)O_T \widehat{x}_3 + \nu E O_T \overline{x}_3 &\in \im[U_S,R_S]+\im B = E\cV^*_{[E,A,B]} + \im B.\\
 \end{aligned}
\]
Assume now inductively that $O_T\overline{x}_3,O_T\widehat{x}_3\in\cV^k_{[E,A,B]}$ (which is trivially satisfied for $k=0$), then
\[\begin{aligned}
   A O_T \overline{x}_3 = \mu E O\overline{x}_3 - \nu E O \widehat{x}_3 + E\overline{v} + B \overline{u}, \\
   A O_T \widehat{x}_3 = \mu E O\widehat{x}_3 + \nu E O \overline{x}_3 + E\widehat{v} + B \widehat{u},
 \end{aligned}
\]
for some $\overline{v},\widehat{v}\in\cV^*_{[E,A,B]}\subseteq \cV^k_{[E,A,B]}$ and $\overline{u},\widehat{u}\in\R^{m}$. Consequently,
\[
   O_T \overline{x}_3 \in A^{-1}(E\cV^k_{[E,A,B]} + \im B) = \cV^{k+1}_{[E,A,B]}\quad \text{and}\quad O_T \widehat{x}_3 \in A^{-1}(E\cV^k_{[E,A,B]} + \im B) = \cV^{k+1}_{[E,A,B]}.
\]
Altogether, we can conclude that $O_T\overline{x}_3,O_T\widehat{x}_3\in \im O_T \cap \cV^*_{[E,A,B]} =\{0\}$, which in view of full column rank of~$O_T$ implies that $x_3 = 0$. Therefore, also $B_{33} u_3 = 0$ which, due to full column rank of $B_{33}$, implies that $u_3 = 0$. This completes the proof.
\end{Proof}

\begin{Remark}[Geometric interpretation of QPFF]
   From Theorem~\ref{thm:QPFF} it becomes clear that the subspace $\cV^*_{[E,A,B]}\cap\cW^*_{[E,A,B]}$ is the reachability/controllability space of~\eqref{eq:EAB} and $\cV^*_{[E,A,B]}$ is the (augmented) consistency space of~\eqref{eq:EAB} (i.e., the set of all initial values $x_0$ for which a (smooth) solution $(x,u)$ of~\eqref{eq:EAB} with $x(0)=x_0$ exists), cf.\ also~\cite[Sec.~6]{BergReis13a}. A matrix representation of the linear system~\eqref{eq:EAB} restricted to $\cV^*_{[E,A,B]}\cap\cW^*_{[E,A,B]}$ is given by $[E_{11},A_{11},B_{11}]$, and a representation~\eqref{eq:EAB} restricted to the quotient-space $\cV^*_{[E,A,B]}/(\cV^*_{[E,A,B]}\cap\cW^*_{[E,A,B]})$ (representing the uncontrollable but consistent states) is given by the uncontrollable ODE system $[E_{22},A_{22},0]$.
\end{Remark}

We stress that condition~\eqref{eq:cond-B-O2} in Theorem~\ref{thm:QPFF} cannot be omitted in general, as the following example shows.

\begin{Example}
Consider the system
\[
    [E,A,B] = \left[\begin{bmatrix}0\\ 1\end{bmatrix}, \begin{bmatrix}1\\ 0\end{bmatrix}, \begin{bmatrix}1\\ 0\end{bmatrix}\right]
\]
and calculate that $\cV^*_{[E,A,B]} = \R$ and $\cW^*_{[E,A,B]}=\{0\}$. Then we may choose $T=R_T =[1]$ and
\[
    S=[R_S,O_S]^{-1} = \begin{bmatrix} 0&1\\ 1&\alpha\end{bmatrix}^{-1} = \begin{bmatrix} -\alpha&1\\1&0\end{bmatrix},\quad \alpha\in\R.
\]
Furthermore, we may choose $V = [1]$ and $F_P = F_2 = [\alpha]$
so that~\eqref{eq:F1_F2_choice} is satisfied.
 Then
\[
    [SET, S(AT+BF_P), SBV] = \left[\begin{bmatrix}1\\ 0\end{bmatrix}, \begin{bmatrix}0\\ 0\end{bmatrix}, \begin{bmatrix}-\alpha\\ 1\end{bmatrix}\right],
\]
which is \textit{not}  in~QPFF~\eqref{eq:QPFF}, if~$\alpha\neq 0$.
 However, we obtain~$\alpha=0$ under the additional condition~\eqref{eq:cond-B-O2}, by which the new system is in~QPFF.
\end{Example}

Finally, we stress that the~QPFF~\eqref{eq:QPFF} is unique in the following sense.

\begin{Proposition}\label{Prop:QPFF-uniqueness}
Consider the system $[E,A,B]\in\Sigma_{\ell,n,m}$ and assume there are $S_1,S_2\in\GL_\ell(\R)$, $T_1,T_2\in\GL_n(\R)$, $V_1,V_2\in\GL_m(\R)$, $F_P^1,F_P^2\in\R^{m\times n}$ such that for $i=1,2$
\[
   [E,A,B] \overset{S_i,T_i,V_i,F^i_P}{\cong_{P}}
   \left[\begin{bmatrix} E^i_{11} & E^i_{12} & E^i_{13} \\ 0 & E^i_{22} & E^i_{23} \\ 0 & 0 & E^i_{33} \end{bmatrix}, \begin{bmatrix} A^i_{11} & A^i_{12} & A^i_{13} \\ 0 & A^i_{22} & A^i_{23} \\ 0 & 0 & A^i_{33} \end{bmatrix}, \begin{bmatrix} B^i_{11} & 0 & B^i_{13} \\ 0 & 0 & 0 \\ 0 & 0 & B^i_{33} \end{bmatrix}\right].
\]
Then the corresponding diagonal blocks (which have the same corresponding sizes as established already in Lemma~\ref{lem:Wong-decoupledQPFF}) are P-feedback equivalent, i.e.\
%\[
%   \color{red} S_1^{-1} S_2 = ..., \quad T_1^{-1} T_2 = ..., \quad V_1^{-1} V_2 = ...,
%\]
%in particular,
$[E^1_{kk},A^1_{kk},B^1_{kk}]\cong_P [E^2_{kk},A^2_{kk},B^2_{kk}]$ for $k=1,2,3$ (with $B^i_{22} := 0_{\ell_2\times m_2}$).
\end{Proposition}

\begin{proof}
% Similar idea as in proof of Prop. 3.3.6 in~DAE-Book (version 12 January 2021), however, with one step.
Consider any system $[E,A,B]$ and the following P-feedback equivalent systems
\[
  [E_{QPFF},A_{QPFF},B_{QPFF}]\overset{S,T,V,F}{\cong_P}[E,A,B]
 \qquad \text{and} \qquad
  [E^W_{QPFF},A^W_{QPFF},B^W_{QPFF}]
  \overset{S^W,T^W,V^W,F^W}{\cong_P}[E,A,B]%\\
%   &[E^{W}_{QPFF},A^{W}_{QPFF},B^{W}_{QPFF}]\overset{S^d,T^d,I,F^d}{\cong_P}[E^W_{QPFF},A^W_{QPFF},B^W_{QPFF}],
\]
where $[E_{QPFF},A_{QPFF},B_{QPFF}]$ is any decoupled~QPFF (not necessarily obtained via the Wong-sequence approach, but probably utilizing Proposition~\ref{Prop:decoupledQPFF}) which is P-feedback equivalent to~$[E,A,B]$,
 and the~QPFF $[E^W_{QPFF},A^W_{QPFF},B^W_{QPFF}]$ is obtained from~$[E,A,B]$ via the Wong-sequence approach (Theorem~\ref{thm:QPFF}). % and the decoupled~QPFF $[E^{W}_{QPFF},A^{W}_{QPFF},B^{W}_{QPFF}]$ is obtained from $[E^W_{QPFF},A^W_{QPFF},B^W_{QPFF}]$ according to Remark~\ref{Prop:decoupledQPFF}.
We will now show that the diagonal blocks of $[E_{QPFF},A_{QPFF},B_{QPFF}]$ are P-feedback equivalent to the corresponding diagonal blocks of $[E^{W}_{QPFF},A^{W}_{QPFF},B^{W}_{QPFF}]$, from which the claim of Proposition~\ref{Prop:QPFF-uniqueness} follows.% (in fact, any (decoupled)~QPFF equivalent to \myred{$[E_2,A_2,B_2]$ is also equivalent to $[E_1,A_1,B_1]$} and its diagonal blocks are therefore P-feedback equivalent to the corresponding diagonal blocks of the~QPFF obtained from $[E_1,A_1,B_1]$ via the Wong-sequence, which in turn are P-feedback equivalent to the diagonal blocks of any~QPFF which is P-feedback equivalent to $[E_1,A_1,B_1]$).

First observe that $[E_{QPFF},A_{QPFF},B_{QPFF}] \overset{\overline{S},\overline{T},\overline{V},\overline{F}}{\cong_P} [E^{W}_{QPFF},A^{W}_{QPFF},B^{W}_{QPFF}]$ with
\[
   \overline{S} = S (S^W)^{-1},\quad
   \overline{T} = (T^W)^{-1} T,\quad
   \overline{V} = (V^W)^{-1} V,\quad
   \overline{F} = (V^W)^{-1} \big(F - F^W \overline{T}\big).
\]
Denote by $\cV^*$ and $\cW^*$ the Wong sequences of the original system $[E,A,B]$.
 %and let $T^{dW} := T^W T^d=[U^{dW}_T,R^{dW}_T,O^{dW}_T]$, $S^{dW} := S^d S^W = [U^{dW}_S,R^{dW}_S,O^{dW}_S]^{-1}$.
Then, by construction (cf.\ Theorem~\ref{thm:QPDFF}),
\[\begin{aligned}
   \im U^{W}_T &= \cV^*\cap\cW^*,& \im[U^{W}_T,R^{W}_T]&=\cV^*, \\
   \im U^{W}_S &= %E\cV^*\cap(A\cW^*+\im B) \overset{\eqref{eq:E(VcapW)_A(VcapW)}}{=}
   E(\cV^*\cap\cW^*), & \im[U^{W}_S,R^{W}_S] &= E\cV^*.
   \end{aligned}
\]
Lemma~\ref{lem:Wong-decoupledQPFF} in conjunction
with Lemma~\ref{lem:Wong-Pfb} yields that the decoupled~QPFF $[E_{QPFF},A_{QPFF},B_{QPFF}]$ satisfies
\[\begin{aligned}
     T^{-1}(\cV^*\cap\cW^*) &= \im \begin{smallbmatrix} I \\ 0 \\ 0 \end{smallbmatrix}, &\quad T^{-1}\cV^* &= \im \begin{smallbmatrix} I & 0\\ 0 & I \\ 0 & 0 \end{smallbmatrix},\\
     E_{QPFF}T^{-1} (\cV^*\cap\cW^*) & =  \im \begin{smallbmatrix} I \\ 0 \\ 0 \end{smallbmatrix},&  E_{QPFF}T^{-1}\cV^* &= \im \begin{smallbmatrix} I & 0\\ 0 & I \\ 0 & 0 \end{smallbmatrix}.
     \end{aligned}
\]
This gives, for some invertible $M_U^T$, $M_R^T$, $M_O^T$, $M_U^S$, $M_R^S$ and~$M_O^S$,
\[\begin{aligned}
    T = [U^{W}_T,R^{W}_T,O^{W}_T] \begin{smallbmatrix} M_U^T & * & * \\ 0 & M_R^T & * \\ 0 & 0 & M_O^T\end{smallbmatrix},\quad
    S^{-1} = [U^{W}_S,R^{W}_S,O^{W}_S] \begin{smallbmatrix} (M_U^S)^{-1} & * & * \\ 0 & (M_R^S)^{-1} & * \\ 0 & 0 & (M_O^S)^{-1}\end{smallbmatrix}.
  \end{aligned}
\]
Therefore,
\[
   \overline{S} = \begin{smallbmatrix} M_U^S & * & * \\ 0 & M_R^S & * \\ 0 & 0 & M_O^S\end{smallbmatrix},\quad \overline{T} = \begin{smallbmatrix} M_U^T & * & * \\ 0 & M_R^T & * \\ 0 & 0 & M_O^T\end{smallbmatrix}
\]
and it follows from $E_{QPFF} = \overline{S} E^{W}_{QPFF} \overline{T}$
 that
\[
    E_{11} = M_U^S E^{W}_{11} M_U^T,\quad
    E_{22} = M_R^S E^{W}_{22} M_R^T,\quad
    E_{33}= M_O^S E^{W}_{33} M_O^T,
\]
where $E_{ii}$, $E^{W}_{ii}$, $i=1,2,3$, are the corresponding diagonal blocks of the block diagonal matrices~$E_{QPFF}$ and~$E^{W}_{QPFF}$.
This shows the desired P-feedback equivalence for the entries of the $E$-matrix.

Writing $\overline{V} = \begin{smallbmatrix} \overline{V}_{11} & \overline{V}_{12} & \overline{V}_{13} \\ \overline{V}_{21} & \overline{V}_{22} & \overline{V}_{23} \\ \overline{V}_{31} & \overline{V}_{32} & \overline{V}_{33} \end{smallbmatrix}$
and  multiplying the equation $B_{QPFF} = \overline{S}B^{W}_{QPFF}\overline{V}$ from the left by~$[0,0,I]$ and from the right by $\begin{smallbmatrix} I & 0 \\ 0 & I \\ 0 & 0 \end{smallbmatrix}$ gives
 $M_O^S B^{W}_{33} [\overline{V}_{31},\overline{V}_{32}] = 0$.
 Invertibility of~$M_O^S$ and full column rank of~$B^{W}_{33}$
 yields~$\overline{V}_{31}=0$ and $\overline{V}_{32}=0$, i.e.
\[
   \overline{V} = \begin{smallbmatrix} \overline{V}_{11} & \overline{V}_{12} & \overline{V}_{13} \\ \overline{V}_{21} & \overline{V}_{22} & \overline{V}_{23} \\ 0 & 0 & \overline{V}_{33} \end{smallbmatrix}\quad\text{with $V_{33}$ invertible}.
\]

%with invertible $\overline{V}_{11}$ and $\overline{V}_{33}$.
Furthermore, from $B_{QPFF} = \overline{S}B^{W}_{QPFF}\overline{V}$ we  see
that $B_{11} = M_U^S B^{W}_{11} \overline{V}_{11}$ and $B_{33} = M_O^S B^{W}_{33} \overline{V}_{33}$ and we need to show that also $\overline{V}_{11}$ is invertible. Assuming that $\overline{V}_{11}$ is not invertible, we find $u_1\in\R^{m_1}\setminus\{0\}$ with $\overline{V}_{11} u_1 = 0$, which implies that $0 = M_U^S B^{W}_{11} \overline{V}_{11} u_1 = B_{11} u_1$ contradicting full column rank of $B_{11}$.

Finally, writing $\overline{F} = \begin{smallbmatrix} \overline{F}_{U,1} & \overline{F}_{R,1} & \overline{F}_{O,1} \\ * & * & * \\ \overline{F}_{U,3} & \overline{F}_{R,3} & \overline{F}_{O,3} \end{smallbmatrix}$ we will show that $\overline{F}_{U,3}$ and $\overline{F}_{R,3}$ are both zero, because then
we have
\[
   \overline{S} B^{W}_{QPFF} \overline{F} = \begin{smallbmatrix} M_U^S & * & * \\ 0 & M_R^S & * \\ 0 & 0 & M_O^S\end{smallbmatrix} \begin{smallbmatrix} B^{W}_{11} \overline{F}_{U,1} &* & * \\
     0 & 0 & 0 \\
   B^{W}_{33} \overline{F}_{U,3} & B^{W}_{33} \overline{F}_{R,3} & B^{W}_{33} \overline{F}_{O,3}. \end{smallbmatrix} = \begin{smallbmatrix} M_U^S B^{W}_{11} \overline{F}_{U,1} &* & * \\
     0 & 0 & * \\
   0 & 0 & M_O^S B^{W}_{33} \overline{F}_{O,3}, \end{smallbmatrix},
\]
which then yields the desired form
\[
A_{QPFF} = \overline{S} (A^{W}_{QPFF} \overline{T} + B^{W}_{QPFF} \overline{F} ) = \begin{smallbmatrix} M_U^S (A^{W}_{11} M_U^T + B^{W}_{11} \overline{F}_{U,1}) & * & * \\ 0 & M_U^S A^{W}_{22} M_U^T & * \\ 0 & 0 & M_O^S(A^{W}_{33} M_O^T + B^{W}_{33} \overline{F}_{O,3}) \end{smallbmatrix}.
\]
From the block structures of $S^W(A T^W + B F^W)$, $\overline{S}$ and $\overline{T}$ it follows that
\[
   \begin{smallbmatrix} * & * & * \\ 0 & * & * \\ 0 & 0 & * \end{smallbmatrix} = \overline{S} S^W(A T^W + B F^W) \overline{T} = SAT + SB F^W \overline{T} = S(AT+BF) + SB(F^W\overline{T}-F),
\]
and due to the block structure of $S(AT+BF)$ we therefore have
\[\begin{aligned}
   \begin{smallbmatrix} * & * & * \\ 0 & * & * \\ 0 & 0 & * \end{smallbmatrix} &= SB(F-F^W\overline{T}) = SBV \overline{V}^{-1} \overline{F} = \begin{smallbmatrix} B_{11} & 0 & 0 \\ 0 & 0 & 0 \\ 0 & 0 & B_{33} \end{smallbmatrix} \begin{smallbmatrix} * & * & * \\ * & * & * \\ 0 & 0 & V_{33}^{-1} \end{smallbmatrix} \begin{smallbmatrix} \overline{F}_{U,1} & \overline{F}_{R,1} & \overline{F}_{O,1} \\ * & * & * \\ \overline{F}_{U,3} & \overline{F}_{R,3} & \overline{F}_{O,3} \end{smallbmatrix}\\
   &= \begin{smallbmatrix} * & * & * \\ 0 & 0 & 0 \\ B_{33} V_{33}^{-1} \overline{F}_{U,3} & B_{33} V_{33}^{-1} \overline{F}_{R,3} & B_{33} V_{33}^{-1} \overline{F}_{O,3} \end{smallbmatrix}.
  \end{aligned}
\]
This shows that $B_{33} V_{33}^{-1} \overline{F}_{U,3}=0$ and $B_{33} V_{33}^{-1} \overline{F}_{R,3}=0$, which in view of the full column rank of $B_{33}$ implies $\overline{F}_{U,3}=0$ and $\overline{F}_{R,3}=0$ as desired.
\end{proof}

\begin{Example}[Example~\ref{Ex:P-form} revisited]\label{Ex:QPFF}
   Consider the system $[E,A,B]$ from Example~\ref{Ex:P-form}. The augmented Wong sequences can be calculated as follows:
   \[\begin{aligned}
       \cV^1_{[E,A,B]} &= \im \begin{smallbmatrix}
          5 & 2 &  0 &  2  \\
          -2 & 0 & -1 & -2  \\
          1 & 0 &  0 &  0  \\
          0 & 1 &  0 &  0  \\
          0 & 0 &  1 &  0  \\
          0 & 0 &  0 &  1  \\
 \end{smallbmatrix} = \cV^*_{[E,A,B]},\\
     \cW^1_{[E,A,B]} &= \im \begin{smallbmatrix}
        3 & -1 & -2 & 0  \\
        -1 &  1 &  0 & 0  \\
        1 &  0 &  0 & 0  \\
        0 &  1 &  0 & 0  \\
        0 &  0 &  1 & 0  \\
        0 &  0 &  0 & 1  \\
     \end{smallbmatrix},\quad
      \cW^2_{[E,A,B]} = \im \begin{smallbmatrix}
        -1 & 2 & 0 & -2 & 0  \\
        1 & 0 & 0 &  0 & 0  \\
        0 & 1 & 0 &  0 & 0  \\
        0 & 0 & 1 &  0 & 0  \\
        0 & 0 & 0 &  1 & 0  \\
        0 & 0 & 0 &  0 & 1  \\
     \end{smallbmatrix} = \cW^*_{[E,A,B]}.
     \end{aligned}
   \]
   Based on these subspaces, we can now easily choose, in virtue of Theorem~\ref{thm:QPFF},
   \[\begin{aligned}
      T &= [U_T,R_T,O_T] = \left[\begin{smallarray}{ccc|c|cc}
          -8 & -5 & 2 & 3 & 1 & 0\\ 4 & 1 & -2 & -1 & 0 & 0\\ -2 & -1 & 0 & 1 & 0 & 0\\ 1 & 0 & 0 & 0 & -1 & -1\\ 0 & 1 & 0 & 1 & 1 & -1\\ 0 & 0 & 1 & -1 & 0 & -1
      \end{smallarray}\right],&\quad
    S^{-1} &= [U_S,R_S,O_S] = \left[\begin{smallarray}{cc|c|cccc}
         1 & 0 & -1 & -1 & -1 & 0 & -1\\ 0 & 1 & -1 & -1 & 1 & -1 & 0\\ 0 & -1 & -1 & 0 & 1 & -1 & 1\\ 1 & 1 & -4 & -2 & 0 & 0 & 0\\ 0 & 1 & -3 & -1 & 1 & 1 & 0\\ -1 & 0 & 1 & \frac{7}{2} & -\frac{7}{2} & -1 & 0\\ 1 & 1 & 0 & 3 & -9 & -1 & -1
      \end{smallarray}\right]\\
     F_P & = [ F_1, F_2 , 0]  = \left[\begin{smallarray}{ccc|c|cc}
    9 & 4 & -1 & -5 & 0 & 0\\ 0 & 0 & 0 & 0 & 0 & 0\\ 6 & 4 & -3 & 1 & 0 & 0  \end{smallarray}\right],&\quad
     V &= [V_1,V_2,V_3] =\left[\begin{smallarray}{c||cc}
        2 & -1 & 0\\ 1 & 1 & 0\\ 0 & -1 & 1
     \end{smallarray}\right],
    \end{aligned}
\]
which results in the QPFF
\[
  [SET, S(AT+BF_P),SBV] = \left[\left[\begin{smallarray}{ccc|c|cc} 3 & 3 & -1 & -5 & -\frac{29}{2} & 33\\ 1 & 0 & -2 & 1 & \frac{5}{2} & 0\\\hline 0 & 0 & 0 & -1 & -\frac{3}{2} & 1\\\hline 0 & 0 & 0 & 0 & -5 & 15\\ 0 & 0 & 0 & 0 & -3 & 9\\ 0 & 0 & 0 & 0 & 0 & 0\\ 0 & 0 & 0 & 0 & 0 & 0 \end{smallarray}\right], \left[\begin{smallarray}{ccc|c|cc} 6 & 5 & 1 & -5 & 18 & -\frac{679}{6}\\ 4 & 3 & -3 & -1 & 6 & -\frac{47}{2}\\\hline 0 & 0 & 0 & -1 & -1 & \frac{5}{3}\\\hline 0 & 0 & 0 & 0 & 15 & -\frac{427}{6}\\ 0 & 0 & 0 & 0 & 7 & -\frac{239}{6}\\ 0 & 0 & 0 & 0 & 2 & -\frac{23}{6}\\ 0 & 0 & 0 & 0 & 1 & -\frac{5}{6}
 \end{smallarray}\right], \left[\begin{smallarray}{c||cc} 1 & 13 & -9\\ 0 & 0 & 0\\\hline 0 & 0 & 0\\\hline 0 & 8 & -5\\ 0 & 6 & -3\\ 0 & 0 & 0\\ 0 & 0 & 0  \end{smallarray}\right]\right].
\]
 In particular, we see that the original system can be decomposed into a completely controllable system in $\Sigma_{2,3,1}$, an uncontrollable ODE system in $\Sigma_{1,1,0}$ and a fully constrained system in $\Sigma_{4,2,2}$.
\end{Example}

\begin{Remark}[Numerical considerations]\label{rem:numerics}
   The calculation of the QPFF relies on obtaining (bases of) kernels and images of matrices, which is inherently numerically unstable because arbitrary small disturbances in the entries of any non-full rank matrix result generically in a full rank matrix. For moderately sized matrices this problem can be circumvented by carrying out the calculations with exact arithmetics (e.g.\ by working with type \texttt{sym} in Matlab as we actually did in Example~\ref{Ex:QPFF}). For large scale sparse systems it is possible to exploit the structural zeros in the calculation of the augmented Wong-sequences and the sparsity can be preserved in the choice of corresponding basis matrices; however, it is a topic of future research to explore whether our approach may be utilized for the analysis of realistic large scale control systems or whether some adjustments or specialized calculation methods are necessary.
\end{Remark}

%============================================================================
\section{PD-feedback forms}\label{Ssec:PDF-form}
%===============================================================

In this section, we investigate PD-feedback which allows for a simpler ``(quasi) canonical'' form compared to the (quasi) P-feedback form since the set of allowed transformations is larger. However, unlike the latter, the quasi PD-feedback form can be derived directly from the Kalman controllability decomposition presented in~\cite{BergTren14} by decomposing the first block row.

\subsection{PD-feedback equivalence}

The following concept of PD-feedback equivalence enlarges the transformation class associated with P-feedback equivalence as in Definition~\ref{Def:PF-equiv}.

\begin{Definition}[PD-feedback equivalence]\label{Def32:PD-equi}
Two systems
$[E_1, A_1, B_1], [E_2, A_2, B_2]  \in \Sigma_{\ell,n,m}$
are called
\emph{PD-feedback equivalent}, if
\begin{equation}\begin{aligned}
&\exists\, S\in \Gl_\ell(\R), T\in \Gl_n(\R), V\in \Gl_m(\R), F_P,F_D\in \R^{m\times n}: \\
&\begin{bmatrix} sE_1-A_1, & B_1\end{bmatrix}
=
S
\begin{bmatrix} sE_2-A_2, & B_2 \end{bmatrix}
\begin{bmatrix} T & 0 \\  sF_D-F_P & V \end{bmatrix}\,;
\end{aligned}\label{eq:PD-feedbequiv}\end{equation}
we write
\begin{align*}
     [E_1,A_1,B_1] \ {\cong_{PD}} \  [E_2, A_2, B_2] \qquad
    \text{or, if necessary,}\qquad  [E_1,A_1,B_1] \ \overset{S,T,V,F_P,F_D}{\cong_{PD}} \  [E_2, A_2, B_2]\,.
\end{align*}
\end{Definition}

\begin{Remark}
Similarly as for P-feedback equivalence, it can easily be shown that PD-feedback equivalence is reflexive, symmetric and transitive, so it is indeed an equivalence relation.
%PD-feedback equivalence is an equivalence relation on $\Sigma_{\ell,n,m}$, which can be verified by observing that for $[E_1,A_1,B_1] \overset{S,T,V,F_P,F_D}{\cong_{PD}}  [E_2, A_2, B_2]$ we have
%\[
%    [E_1,A_1,B_1] \ \overset{S^{-1},T^{-1},V^{-1},-V^{-1}F_PT^{-1},-V^{-1}F_DT^{-1}}{\cong_{PD}} \  [E_2, A_2, B_2],
%\]
%because
%\[
%    \begin{bmatrix} T & 0\\ sF_D - F_P & V\end{bmatrix}^{-1} = \begin{bmatrix} T^{-1} & 0\\-V^{-1}(sF_D-F_P)T^{-1} & V^{-1}\end{bmatrix}.
%\]
\end{Remark}

For later use we show how the augmented Wong sequences change under PD-feedback.

\begin{Lemma}[Augmented Wong sequences under PD-feedback]\label{lem:Wong-PDfb}
If the systems  $[E_1, A_1, B_1], [E_2, A_2, B_2] \in \Sigma_{\ell,n,m}$
are PD-feedback equivalent $[E_1, A_1, B_1] \ \overset{S,T,V,F_P,F_D}{\cong_{PD}} \  [E_2, A_2, B_2]$, then
\begin{equation*}\label{eq:rel-wong}
\begin{aligned}
    \forall\, i\in\N_0:\quad \cV_{[E_1,A_1,B_1]}^i = T^{-1} \cV_{[E_2,A_2,B_2]}^i \quad \text{and}\quad
      \cW_{[E_1,A_1,B_1]}^i = T^{-1} \cW_{[E_2,A_2,B_2]}^i.
    \end{aligned}
\end{equation*}
\end{Lemma}
\begin{proof} We prove the statement by induction. It is clear that $\cV_{[E_1,A_1,B_1]}^0 = T^{-1} \cV_{[E_2,A_2,B_2]}^0$.
Assume that $\cV_{[E_1,A_1,B_1]}^i = T^{-1} \cV_{[E_2,A_2,B_2]}^i$ for some $i\geq 0$.
Then~\eqref{eq:PD-feedbequiv} yields
\begin{align*}
     \cV_{[E_1,A_1,B_1]}^{i+1}   &= A_1^{-1} (E_1 \cV_{[E_1,A_1,B_1]}^i + \im  B_1)\\
    &= \setdef{x\in\R^n}{ \begin{array}{l} \exists\, y\in\cV_{[E_1,A_1,B_1]}^i\ \exists\, u\in\R^m:\\[1.5mm] (SA_2 T + SB_2F_P) x = (SE_2T + SB_2F_D) y + SB_2V u\end{array}}\\
    &= \setdef{x\in\R^n}{ \exists\, z\in\cV_{[E_2,A_2,B_2]}^i\ \exists\, v\in\R^m:\ A_2 T x = E_2 z + B_2 v}\\
    &= T^{-1} \left( A_2^{-1} (E_2 \cV_{[E_2,A_2,B_2]}^i + \im  B_2)\right) = T^{-1} \cV_{[E_2,A_2,B_2]}^{i+1}.
\end{align*}
The proof of the statements  about~$\cW_{[E_1,A_1,B_1]}^i$ and~$\cW_{[E_2,A_2,B_2]}^i$ is similar and omitted.
\end{proof}

\subsection{PD-feedback form (PDFF)}

\begin{Definition}[PD-feedback form]\label{def:PDFF}
The system $[E,A,B]\in\Sigma_{\ell,n,m}$
 is said to be in
\textit{PD-feedback form} (PDFF), if
\begin{equation}
 \label{eq:PDform}
[E, A,B]
\ = \
\left[
  \begin{bmatrix} K_{\bsalpha} & 0 & 0 & 0\\ 0 & I_{n_{\overline{c}}} & 0
  & 0 \\ 0 & 0 & N_\bsbeta & 0  \\ 0 & 0 & 0 & K_{\bsgamma}^\top \\ 0&0&0&0
 \end{bmatrix},
  \begin{bmatrix}L_{\bsalpha} & 0 & 0 & 0\\ 0 & A_{\overline{c}} & 0 & 0
   \\ 0 & 0 & I_{|\bsbeta|} & 0 \\ 0 & 0 & 0 & L_{\bsgamma}^\top \\ 0&0&0&0 \end{bmatrix},
  \begin{bmatrix} 0&0  \\ 0&0 \\ 0&0\\ 0&0\\ 0 & I_{r} \end{bmatrix}
\right],
\end{equation}
where $r=\rk  B$,
$\bsalpha\in\N^{n_\bsalpha},\bsbeta\in\N^{n_\bsbeta},\bsgamma\in\N^{n_\bsgamma}$
denote multi-indices, and  $A_{\overline{c}}\in\R^{n_{\overline{c}}\times n_{\overline{c}}}$.
\end{Definition}

Two PD-feedback equivalent systems have the same~PDFF up to permutation of the entries of $\bsalpha, \bsbeta, \bsgamma$ and similarity of $A_{\overline{c}}$.

\begin{Proposition}[Uniqueness of indices for~PDFF]\label{Prop:Indices-uniqueness-PDFF}
Let $[E_i,A_i,B_i]\in\Sigma_{\ell,n,m}$, $i=1,2$,
be in~PDFF~\eqref{eq:PDform} with corresponding $r_i=\rk  B_i$, multi-indices
$\bsalpha_i\in\N^{n_{\bsalpha_i}},\bsbeta_i\in\N^{n_{\bsbeta_i}},\bsgamma_i\in\N^{n_{\bsgamma_i}}$, and  $A_{\overline{c},i}\in\R^{n_{\overline{c},i}\times n_{\overline{c},i}}$.
If $[E_1, A_1, B_1] \cong_{PD} [E_2, A_2, B_2]$, then
\[
    \bsalpha_1 =  P_\bsalpha \bsalpha_2,\quad \bsbeta_1 = P_\bsbeta  \bsbeta_2,\quad \bsgamma_1 =  P_\bsgamma \bsgamma_2,
\]
and
\[
    r_1 = r_2,\quad n_{\overline{c},1} = n_{\overline{c},2}, \quad A_{\overline{c},1} = H^{-1} A_{\overline{c},2} H,
\]
for permutation matrices $P_\bsalpha, P_\bsbeta, P_\bsgamma$    of appropriate sizes
and~$H\in\Gl_{n_{\overline{c},1}}(\R)$.
\end{Proposition}

\begin{Proof} It follows from~\eqref{eq:PD-feedbequiv} that $B_1 = SB_2V$ and hence $r_1 = r_2$. Furthermore, $sE_1-A_1=S(sE_2-A_2)T+SB_2(sF_D-F_P)$, which gives
\[
    \begin{smallbmatrix} sK_{\bsalpha_1}-L_{\bsalpha_1} & 0 & 0 & 0\\* 0 & sI_{n_{\overline{c},1}}- A_{\overline{c},1}  & 0
  & 0 \\ 0 & 0 & sN_{\bsbeta_1}-I_{|\bsbeta_1|} & 0  \\ 0 & 0 & 0 & sK_{\bsgamma_1}^\top-L_{\bsgamma_1}^\top \\ 0&0&0&0
 \end{smallbmatrix}
    =  S \begin{smallbmatrix} sK_{\bsalpha_2}-L_{\bsalpha_2} & 0 & 0 & 0\\* 0 & sI_{n_{\overline{c},2}}- A_{\overline{c},2}  & 0
  & 0 \\ 0 & 0 & sN_{\bsbeta_2}-I_{|\bsbeta_2|} & 0  \\ 0 & 0 & 0 & sK_{\bsgamma_2}^\top-L_{\bsgamma_2}^\top \\ sF_1-G_1 & sF_2-G_2 & sF_3-G_3 & sF_4-G_4
 \end{smallbmatrix} T
\]
for some matrices $F_i$ and~$G_i$ of appropriate sizes. Set $r:=r_1=r_2$ and write
\[
    S = \begin{bmatrix} S_{11} & S_{12}\\ S_{21} & S_{22}\end{bmatrix},\quad S_{22}\in\R^{r\times r},
\]
and $S_{11}, S_{12}, S_{21}$ of appropriate size. Then $B_1 = SB_2V$
yields
\[
    \begin{bmatrix} 0 & 0\\ 0 & I_r\end{bmatrix} = \begin{bmatrix} S_{11} & S_{12}\\ S_{21} & S_{22}\end{bmatrix} \begin{bmatrix} 0 & 0\\ 0 & I_r\end{bmatrix} V = \begin{bmatrix} 0 & S_{12}\\ 0 & S_{22}\end{bmatrix} V,
\]
and hence $S_{22}\in\Gl_r(\R)$ and $S_{12}=0$. Then again $S_{11}\in\Gl_{\ell-r}(\R)$ and we have
\[
    \begin{smallbmatrix} sK_{\bsalpha_1}-L_{\bsalpha_1} & 0 & 0 & 0\\* 0 & sI_{n_{\overline{c},1}}- A_{\overline{c},1} & 0
  & 0 \\ 0 & 0 & sN_{\bsbeta_1}-I_{|\bsbeta_1|} & 0 \\ 0 & 0 & 0 & sK_{\bsgamma_1}^\top-L_{\bsgamma_1}^\top \end{smallbmatrix}\\
    =  S_{11} \begin{smallbmatrix} sK_{\bsalpha_2}-L_{\bsalpha_2} & 0 & 0 & 0\\* 0 & sI_{n_{\overline{c},2}}- A_{\overline{c},2}  & 0
  & 0 \\ 0 & 0 & sN_{\bsbeta_2}-I_{|\bsbeta_2|} & 0  \\ 0 & 0 & 0 & sK_{\bsgamma_2}^\top-L_{\bsgamma_2}^\top
 \end{smallbmatrix} T,
\]
which in view of~\cite[Rem.~2.8]{BergTren13} implies the assertion.
\end{Proof}

We are now in the position to show that any $[E,A,B]\in\Sigma_{\ell,n,m}$ is PD-feedback equivalent to a system in~PDFF.
This result has already been observed in the seminal work~\cite{LoisOzca91} co-authored by Nicos Karcanias; it simply consists of a left transformation~$S$ together with an input space transformation~$V$ which puts the matrix~$B$ into the
form $SBV=\left[\begin{smallmatrix} 0&0\\ 0&I_r\end{smallmatrix}\right]$ and an additional transformation which puts the pencil $sNE-NA$,
where $N=[I_{l-r},0]W$, into Kronecker canonical form.

\begin{Theorem}[PDFF~\cite{LoisOzca91}]\label{Thm32:PDform}
For any system $[E,A,B]\in\Sigma_{\ell,n,m}$
  there exist $S\in \Gl_\ell(\R), T\in \Gl_n(\R), V\in \Gl_m(\R), F_P, F_D\in \R^{m\times n}$
such that
\[
    [SET+SBF_D,SAT+SBF_P,SBV]
\ \
\text{is in~PDFF~\eqref{eq:PDform}.}
\]
\end{Theorem}

\begin{Remark}[PDFF from PFF]\label{rem:PFF->PDFF}
We may directly derive the~PDFF~\eqref{eq:PDform} from the PFF~\eqref{eq:Pform}. To this end,
observe that the system $[I_{\beta_i},N_{\beta_i}^\top,e_{\beta_i}^{[\beta_i]}]$ can be written as
\[
    \ddt L_{\beta_i}  x_{c[i]}(t) = K_{\beta_i} x_{c[i]}(t),\quad \ddt x_{c[i],\beta_i}(t) = u_{c[i]}(t),
\]
and hence it is, after a permutation of the states and the equations, PD-feedback equivalent to the system
\[\left[ \begin{bmatrix} K_{\beta_i}\\ 0\end{bmatrix}, \begin{bmatrix} L_{\beta_i}\\ 0\end{bmatrix}, \begin{bmatrix} 0\\ 1\end{bmatrix}\right].\]
On the other hand, the system $[K_{\kappa_i}^\top,L_{\kappa_i}^\top,e_{\kappa_i}^{[\kappa_i]}]$ can be written as
\[
    \ddt N_{\kappa_i-1} x_{ob[i]}(t) = x_{ob[i]}(t),\quad \ddt x_{ob[i],\kappa_{i-1}}(t) = u_{ob[i]}(t),
\]
and hence it is PD-feedback equivalent to the system
\[\left[ \begin{bmatrix} N_{\kappa_i-1} \\ 0\end{bmatrix}, \begin{bmatrix} I_{\kappa_i-1}\\ 0\end{bmatrix}, \begin{bmatrix} 0\\ 1\end{bmatrix}\right].\]
It is now easy to see that we obtain a system in the form~\eqref{eq:PDform} after some block permutations.
\end{Remark}

\begin{Example}[PDFF]\label{Ex:PD-form}
We revisit Example~\ref{Ex:P-form} to illustrate Theorem~\ref{Thm32:PDform} and consider again the system $[E,A,B]\in\Sigma_{7,6,3}$. Utilizing Remark~\ref{rem:PFF->PDFF} we can easily chose
\[
   \begin{aligned}
        S &= \begin{smallbmatrix} -15 & 2 & 4 & 5 & -6 & -6 & 4\\ -3 & -1 & 0 & 3 & -2 & 0 & 0\\ 8 & 2 & 0 & -5 & 4 & 2 & -1\\ -1 & 0 & 0 & 1 & -1 & 0 & 0\\ -16 & -1 & 2 & 9 & -8 & -5 & 3\\ -6 & 0 & 1 & 3 & -3 & -2 & 1\\ -4 & 0 & 1 & 2 & -2 & -1 & 1
\end{smallbmatrix},&\quad
       T &= \begin{smallbmatrix}
         -17 & 10 & -13 & -3 & -8 & 6\\ 13 & -6 & 9 & 2 & 6 & -4\\ -7 & 4 & -5 & -1 & -3 & 2\\ 6 & -3 & 4 & 1 & 3 & -2\\ -5 & 2 & -3 & 0 & -2 & 1\\ 3 & -2 & 2 & 0 & 1 & -1
       \end{smallbmatrix},\\
       F_P &= \begin{smallbmatrix}
         3 & 3 & -2 & -2 & 0 & 2\\ -14 & 10 & -12 & -3 & -7 & 6\\ 7 & -4 & 6 & 3 & 4 & -3
       \end{smallbmatrix},&\quad
       F_D &= \begin{smallbmatrix}
          0 & 0 & -2 & 0 & 0 & 0\\ 0 & 0 & -1 & 0 & 0 & 0\\ 0 & 0 & 0 & 0 & 0 & 1
       \end{smallbmatrix},&\quad
       V &= \begin{smallbmatrix}
           2 & 0 & -1\\ 1 & 0 & -1\\ 0 & -1 & -1
       \end{smallbmatrix},
   \end{aligned}
\]
which results in
\[\begin{aligned}
    S(ET+BF_D) &= \diag(K_1,K_2,I_1,N_1,N_1,0_{3\times 0}),\\ S(AT+BF_P) &= \diag(L_1,L_2,A_{\overline{c}},I_1,I_1,0_{3\times 0}),\\
    SBV &= \diag(0_{4\times 0},I_3),
    \end{aligned}
\]
with $A_{\overline{c}}=[1]$. In particular, $[E,A,B]$ is PD-feedback equivalent to a PDFF \eqref{eq:PDform} with $\bsalpha = (1,2)$, $n_{\overline{c}}=1$, $\bsbeta=(1,1)$, $n_\bsgamma=0$, $r = 3$.
\end{Example}

%_________________________________________________________________________
\subsection{Quasi PD-feedback form (QPDFF)}

We will now weaken the PD-feedback form to a \textit{quasi} PD-feedback form, again exploiting the augmented Wong sequences as the crucial tool to achieve the necessary geometric insight.

\begin{Definition}\label{def:QPDFF}
  A system $[E,A,B]\in\Sigma_{\ell,n,m}$ is said to be in \emph{quasi PD-feedback form (QPDFF)}, if
  \begin{equation}\label{eq:QPDFF}
     [E,A,B] = \left[\begin{bmatrix} E_{11} & E_{12} & E_{13} \\ 0 & E_{22} & E_{23} \\ 0 & 0 & E_{33}\\ 0&0&0 \end{bmatrix}, \begin{bmatrix} A_{11} & A_{12} & A_{13} \\ 0 & A_{22} & A_{23} \\ 0 & 0 & A_{33}\\ 0&0&0  \end{bmatrix}, \begin{bmatrix} 0&0 \\ 0 & 0 \\ 0&0\\ 0 & \hat B \end{bmatrix}\right],
  \end{equation}
  where
  \begin{enumerate}[(i)]
     \item\label{item:QPDFF(i)} $E_{11}, A_{11}\in\R^{\ell_1\times n_{1}}$ with $\ell_1 < n_1$ and $\rk E_{11} = \rk_\C [\lambda E_{11} - A_{11}] = \ell_1$ for all $\lambda\in\C$,
     \item\label{item:QPDFF(ii)} $E_{22},A_{22}\in\R^{\ell_{2}\times n_{2}}$ with $\ell_2 = n_2$ and $E_{22}\in\Gl_{n_2}(\R)$,
     \item\label{item:QPDFF(iii)} $E_{33}, A_{33} \in\R^{\ell_{3}\times n_{3}}$ satisfy $\rk_\C (\lambda E_{33}-A_{33}) = n_{3}$ for all $\lambda\in\C$,
     \item $\hat B \in\Gl_{m_2}(\R)$ for $m_2=\rk B$
  \end{enumerate}
  and the remaining matrices have suitable sizes. Furthermore, we call a~QPDFF \emph{decoupled},  if all off-diagonal blocks are zero.
\end{Definition}

The control theoretic interpretation of the~QPDFF is that the system can be decomposed in four parts. The first three parts contain only state variables, they form a homogeneous~DAE. The first part consists of an underdetermined~DAE which is completely controllable\footnote{In this context, it may be better to speak of ``complete reachability'' instead of ``complete controllability'', because it seems rather unintuitive to call a system which is not affected by an input ``controllable''; nevertheless, this naming convention also occurs in the context of behavioral controllability, where the trivial behavior consisting of all trajectories is also called ``controllable''.}, the second part is actually an uncontrollable~ODE and the third part is a~DAE which has only the trivial solution (and is therefore trivially behaviorally controllable). The fourth part contains  only
the input~$u=(u_1,u_2)$, where the first component of the input is completely free (but does not influence the state) and the second input component is maximally constrained~$(0=u_2)$.

\begin{Remark}
   Utilizing Sylvester equations in a very similar way as in Proposition~\ref{Prop:decoupledQPFF} it can be shown that any~QPDFF is PD-feedback equivalent to a decoupled~QPDFF with identical diagonal blocks. Since the input and state are completely decoupled in the~QPDFF this decoupling is actually much easier to achieve than the decoupling in the~QPFF.
\end{Remark}

For later use we derive some properties of the augmented Wong limits for a system $[E,A,B]$ which is in decoupled QPDFF~\eqref{eq:QPDFF}.

\begin{Lemma}\label{Lem:Wong-QPDFF}
Assume $[E,A,B]\in\Sigma_{\ell,n,m}$ is in decoupled~QPDFF~\eqref{eq:QPDFF}, then
\begin{equation}\label{eq:QPDFF-Wong}
   \cV^*_{[E,A,B]}\cap\cW^*_{[E,A,B]} = \R^{n_1}\times\{0\}^{n_2+n_3},\quad \cV^*_{[E,A,B]} = \R^{n_1+n_2} \times \{0\}^{n_3}
\end{equation}
and
\[
\begin{array}{rcl}
\im B &=&              \{0\}^{\ell_1 + \ell_2+\ell_3} \times \R^{m_2},\\
   E(\cV^*_{[E,A,B]} \cap \cW^*_{[E,A,B]}) + \im B
   &=& \R^{\ell_1}\times \{0\}^{\ell_2+\ell_3} \times \R^{m_2},\\
   E\cV^*_{[E,A,B]} + \im B &= &\R^{\ell_1+\ell_2}\times\{0\}^{\ell_3} \times \R^{m_2}.
\end{array}
\]
In particular (in view of Lemma~\ref{lem:Wong-PDfb}), two~QPDFFs which are PD-feedback equivalent have the same block sizes in $E$, $A$ and $B$.
\end{Lemma}

The proof  utilizes the observation that
\[
    \cV^*_{[E,A,B]} = \cV^*_{[E,A,0]} \quad \text{and}\quad \cW^*_{[E,A,B]} = \cW^*_{[E,A,0]}
\]
and  is very similar to the proof of Lemma~\ref{lem:Wong-decoupledQPFF}.
It is therefore omitted.

%\myblue{
%A detailed proof is carried out in the Appendix, however
%due to block structure of the~QPFF~\eqref{eq:QPDFF} it is easily seen that augmented Wong-sequences coincide with the classical Wong-sequences, in particular,}
%\[
%    \cV^*_{[E,A,B]} = \cV^*_{[E,A,0]} \quad \text{and}\quad \cW^*_{[E,A,B]} = \cW^*_{[E,A,0]}.
%\]
%\myblue{
%Furthermore, it is easily seen that any block triangular pencil $sE-A$ satisfying the conditions \ref{item:QPDFF(i)}--\ref{item:QPDFF(iii)} of the~QPDFF is in fact in Quasi-Kronecker form (with combined nilpotent and overdetermined part), cf.~\cite{BergTren13}, consequently, the properties~\eqref{eq:QPDFF-Wong} are a direct consequence from the corresponding properties of the classical Wong sequence of the matrix pencil $sE-A$. However, an explicit statement of the properties~\eqref{eq:QPDFF-Wong} for the quasi-Kronecker form is not yet available in the literature.
%}

We will now show that any system $[E,A,B]$ is PD-feedback equivalent to a system in~QPDFF and that the transformation matrices can be obtained from the augmented Wong sequences; this provides some geometric insight in the decoupling.

\begin{Theorem}[Quasi PD-feedback form]\label{thm:QPDFF}
 Consider $[E,A,B]\in\Sigma_{\ell,n,m}$ with corresponding augmented Wong limits $\cV^*_{[E,A,B]}$ and $\cW^*_{[E,A,B]}$. Choose full column rank matrices $U_T\in\R^{n\times n_{1}}$, $R_T\in\R^{n\times n_{2}}$, $O_T\in\R^{n\times n_{3}}$, $Q_S\in\R^{\ell\times m_2}$, $U_S\in\R^{\ell\times \ell_{1}}$, $R_S\in\R^{\ell\times \ell_{2}}$, $O_S\in\R^{\ell\times \ell_{3}}$ such that
 \[\begin{aligned}
    \im U_T &= \cV^*_{[E,A,B]}\cap\cW^*_{[E,A,B]},\\
    \im R_T \oplus \im U_T &= \cV^*_{[E,A,B]}, \\
    \im O_T\oplus \im R_T \oplus \im U_T &= \R^n, \\[2ex]
    \im Q_S &= \im B,\\
    \im U_S \oplus \im Q_S &= E(\cV^*_{[E,A,B]} \cap \cW^*_{[E,A,B]})+\im B,\\
    \im R_S \oplus \im U_S \oplus \im Q_S &= E\cV^*_{[E,A,B]} + \im B,\\
    \im O_S \oplus \im R_S \oplus \im U_S \oplus \im Q_S &= \R^{\ell}.
    \end{aligned}
 \]
 Let $T:=[U_T,R_T,O_T]$, $S:=[U_S,R_S,O_S,Q_S]^{-1}$ and choose (not necessarily full rank) matrices $F_P, F_D\in\R^{m\times n}$ such that
 \[
    0 = [0, I_{m_2}] S (E T + B F_D),\quad
    0 = [0, I_{m_2}] S (A T + B F_P),
 \]
 and choose $V=[V_1,V_2]$ with full column rank matrices $V_{1}\in\R^{m\times m_1}$, $V_2\in\R^{m\times m_2}$ such that
 \[
     \im V_{1} = \ker B,\quad \im V_{1} \oplus \im V_2 = \R^m.
 \]
 Then $[S(ET+BF_D),S(AT+BF_P),SBV]$ is in~QPDFF~\eqref{eq:QPDFF}.
\end{Theorem}
\begin{Proof} \emph{Step 1}:  We show that the block structure of the~QPDFF~\eqref{eq:QPDFF} is achieved. The subspace inclusions~\eqref{eq:AWS_invariance} and the equalities~\eqref{eq:EVBcapAWB} imply that
\[
  \begin{aligned}
     \im E U_T &\subseteq \im [U_S,Q_S],&\quad \im A U_T &\subseteq \im [U_S,Q_S],\\
     \im E R_T &\subseteq \im [U_S,R_S,Q_S],&\quad \im A R_T &\subseteq \im  [U_S,R_S,Q_S],\\
     \im E O_T &\subseteq \im [U_S,R_S,O_S,Q_S]=\R^n,&\quad \im A O_T &\subseteq \im [U_S,R_S,O_S,Q_S]=\R^\ell,
  \end{aligned}
\]
hence there exists matrices $E_{11}, E_{12}, E_{13}, E_{22}, E_{23}, E_{33}, A_{11}, A_{12}, A_{13}, A_{22}, A_{23}, A_{33}$, $F_{1}^E$, $F_{2}^E$, $F_3^E$, $F_{1}^A$, $F_{2}^A$, $F_3^A$ such that
\[\begin{aligned}
   E U_T &= U_S E_{11} + Q_S F_1^E, &\quad A U_T &= U_S A_{11}  + Q_S F_1^A,\\
   E R_T &= U_S E_{12} + R_S E_{22} +  Q_S F_2^E,& A R_T &= U_S A_{12} + R_S A_{22}  + Q_S F_2^A,\\
   E O_T &= U_S E_{13} + R_S E_{23} + O_S E_{33}  + Q_S F_3^E, &    A O_T &= U_S A_{13} + R_S A_{23} + O_S A_{33} + Q_S F_3^A.
   \end{aligned}
\]
Furthermore, $B = Q_S \bar B$ for some $\bar B \in\R^{m_2\times m}$ and $\hat B := \bar B V_2 \in\R^{m_2\times m_2}$ since $\rk V_2 = m - \dim \ker B = \rk B = m_2$. Then $0 = \hat B x = \bar B V_2 x$
yields
 $V_2 x \in \ker \bar B \cap \im V_2 = \ker B \cap \im V_2 = \{0\}$,
 thus $\hat B \in\Gl_{m_2}(\R)$. Therefore,
\[
    SBV = S [0, BV_2] = S [U_S,R_S,O_S,Q_S] \begin{bmatrix} 0&0\\ 0&0\\ 0&0\\ 0&\bar B V_2\end{bmatrix} = \begin{bmatrix} 0&0\\ 0&0\\ 0&0\\ 0&\hat B\end{bmatrix}
\]
has the block structure as in~\eqref{eq:QPDFF}. Set $F^E := [F_1^E, F_2^E, F_3^E]$ and $F^A :=[F_1^A, F_2^A, F_3^A]$, then
$\im [0, I_{m_2}] S B = \im [0, \hat B] = \R^{m_2}$ and hence the equations
\begin{align*}
  F^E  + [0, I_{m_2}] S B F_D = 0,\quad F^A  + [0, I_{m_2}] S B F_P = 0
\end{align*}
have solutions $F_D, F_P\in\R^{m\times n}$ which satisfy
\[
    [0, I_{m_2}] S (E T + B F_D) = F^E + [0, \hat B] F_D = 0,\quad [0, I_{m_2}] S (A T + B F_P) = F^A + [0, \hat B] F_P = 0.
\]
This proves that $S (E T + B F_D)$ and $S (A T + B F_P)$ have the block structure as in~\eqref{eq:QPDFF}.
\\

\noindent
\emph{Step 2}: We show that $E_{11}, A_{11}$ satisfy Definition~\ref{def:QPDFF}\,\ref{item:QPDFF(i)}.
\\
 Denote by $\cV^i_{[E_{11},A_{11},0]}$, $\cW^i_{[E_{11},A_{11},0]}$, $\cV^*_{[E_{11},A_{11},0]}$, $\cW^*_{[E_{11},A_{11},0]}$ the Wong sequences and Wong limits corresponding to the matrix pencil ${sE_{11}-A_{11}}$. By choice of~$U_T$ we have
\[
   \cV^*_{[E,A,B]}\cap \cW^*_{[E,A,B]} = \im U_T = T (\R^{n_1}\times \{0\}^{n_2+n_3}).
\]
It follows from Lemma~\ref{lem:Wong-PDfb} that for $[\tilde E, \tilde A, \tilde B] := [S (E T + B F_D),  S (A T + B F_P), SBV]$ we have
\[
    \cV^*_{[E,A,B]} = T\cV^*_{[\tilde E, \tilde A, \tilde B]},\quad \cW^*_{[E,A,B]} = T\cW^*_{[\tilde E,\tilde A, \tilde B]},
\]
hence
\[
    \cV^*_{[\tilde E, \tilde A, \tilde B]} \cap \cW^*_{[\tilde E,\tilde A, \tilde B]} = \R^{n_1}\times \{0\}^{n_2+n_3}.
\]
Now let $x\in \cV^*_{[\tilde E, \tilde A, \tilde B]} \cap \cW^*_{[\tilde E, \tilde A, \tilde B]}$.
Then $x= (x_1^\top, 0, 0)^\top$ for some~$x_1\in\R^{n_1}$ and
\begin{align*}
    \tilde A  x \in \tilde A\big( \cV^*_{[\tilde E, \tilde A, \tilde B]} \cap \cW^*_{[\tilde E, \tilde A, \tilde B]}\big) + \im \tilde B
    \overset{\eqref{eq:EVBcapAWB}}{=} \tilde E\big( \cV^*_{[\tilde E, \tilde A, \tilde B]} \cap \cW^*_{[\tilde E, \tilde A, \tilde B]}\big) + \im \tilde B,
\end{align*}
and thus there exist $y = (y_1^\top, 0, 0)^\top\in \cV^*_{[\tilde E, \tilde A, \tilde B]} \cap \cW^*_{[\tilde E, \tilde A, \tilde B]}$ and $u_2\in\R^{m_2}$ such that $A_{11} x_1 = E_{11} y_1$ and $0 = \hat B u_2$, thus $u_2 =0$. This implies $x_1\in A_{11}^{-1}(E_{11} \{y_1\}) \subseteq \cV^1_{[E_{11},A_{11},0]}$.
A similar reasoning  yields
%With an analogous argument we may deduce
$y_1\in\cV^1_{[E_{11},A_{11},0]}$,
and therefore
%but this implies
\[
    x_1\in A_{11}^{-1}(E_{11} \{y_1\}) \subseteq  A_{11}^{-1}(E_{11} \cV^1_{[E_{11},A_{11},0]}) \subseteq \cV^2_{[E_{11},A_{11},0]}.
\]
Again, we conclude similarly that
 $y_1\in \cV^2_{[E_{11},A_{11},0]}$ and thus $x_1\in\cV^3_{[E_{11},A_{11},0]}$. Proceeding in this way we obtain $x_1\in\cV^*_{[E_{11},A_{11},0]}$.

Now let $j^*,r^*\in\N$ be such that $\cW^*_{[\tilde E, \tilde A, \tilde B]} = \cW^{j^*}_{[\tilde E, \tilde A, \tilde B]}$ and
$\cW^*_{[E_{11},A_{11},0]} = \cW^{r^*}_{[E_{11},A_{11},0]}$ and set
$q^* := \max\{j^*,r^*\}$.
Since $x \in \cW^*_{[\tilde E, \tilde A, \tilde B]} = \cW^{q^*}_{[\tilde E, \tilde A, \tilde B]}$, it follows from Definition~\ref{Def:AWS} that there exist $y_k\in\cW^{k}_{[\tilde E, \tilde A, \tilde B]}$ and $u_k\in\R^m$, $k=0,\ldots,q^*-1$,
 such that $\tilde Ex = \tilde A y_{q^*-1} + \tilde B u_{q^*-1}$
 and $\tilde Ey_k = \tilde Ay_{k-1} + \tilde B u_{k-1}$ for all $k=0,\ldots,q^*-1$.
 We find that
\begin{align*}
    \tilde A y_{q^*-1} = \tilde E x - \tilde B u_{q^*-1} \in \tilde A \big( \cV^*_{[\tilde E, \tilde A, \tilde B]} \cap \cW^*_{[\tilde E, \tilde A, \tilde B]}\big) + \im \tilde B
    \overset{\eqref{eq:EVBcapAWB}}{=} \tilde A\big( \cV^*_{[\tilde E, \tilde A, \tilde B]} \cap \cW^*_{[\tilde E, \tilde A, \tilde B]}\big) + \im \tilde B,
\end{align*}
thus
\begin{align*}
    y_{q^*-1} \in \tilde A^{-1}\Big(\tilde A\big( \cV^*_{[\tilde E, \tilde A, \tilde B]} \cap \cW^*_{[\tilde E, \tilde A, \tilde B]}\big) + \im \tilde B\Big)
    = \cV^*_{[\tilde E, \tilde A, \tilde B]} \cap \cW^*_{[\tilde E, \tilde A, \tilde B]} + \tilde A^{-1}(\im \tilde B) \subseteq \cV^*_{[\tilde E, \tilde A, \tilde B]}.
\end{align*}
Therefore,
\[
  y_{q^*-1} \in \cV^*_{[\tilde E, \tilde A, \tilde B]} \cap \cW^{q^*-1}_{[\tilde E, \tilde A, \tilde B]}.
\]
Analogously, we may show that
\[
  y_{k} \in \cV^*_{[\tilde E, \tilde A, \tilde B]}\cap \cW^{k}_{[\tilde E, \tilde A, \tilde B]},\quad k=0,\ldots,q^*-2.
\]
This implies that $y_k = (y_{k,1}^\top, 0, 0)^\top$ for some $y_{k,1}\in\R^{n_1}$, $k=0,\ldots,q^*-1$, and hence, in particular,
\[
    E_{11} y_{1,1} = 0,\ E_{11} y_{2,1} = A_{11} y_{1,1},\ \ldots,\ E_{11} y_{q^*-1,1} = A_{11} y_{q^*-2,1},\ E_{11} x_1 = A_{11} y_{q^*-1,1}.
\]
Therefore, we obtain $x_1\in\cW^*_{[E_{11},A_{11},0]}$. This proves
\begin{align*}
    \big(\cV^*_{[E_{11},A_{11},0]}\cap \cW^*_{[E_{11},A_{11},0]}\big) \times \{0\}^{n_2+n_3} &=  \cV^*_{[\tilde E, \tilde A, \tilde B]} \cap \cW^*_{[\tilde E, \tilde A, \tilde B]}
    = \R^{n_1}\times \{0\}^{n_2+n_3}.
\end{align*}

To conclude Step~2, it remains to show that $\rk E_{11} = \ell_1$, then the assertion follows from~\cite[Thm.~2.3]{BergTren12}. To this end, observe that $U_S E_{11} + Q_S F_1^E = E U_T$ and $Q_S \bar B = B$ imply that
\[
    \im [U_S, Q_S] \begin{bmatrix} E_{11} & 0\\ F_1^E & \bar B\end{bmatrix} = \im E U_T + \im B = E(\cV^*_{[E,A,B]} \cap \cW^*_{[E,A,B]})+\im B = \im [U_S, Q_S].
\]
As a consequence, full column rank of $[U_S, Q_S]$ gives that $\begin{smallbmatrix} E_{11} & 0\\ F_1^E & \bar B\end{smallbmatrix}$ has full row rank, by which $\rk E_{11} = \ell_1$.
\\

\noindent
\emph{Step 3}: The proof of $\ell_2=n_2$, invertibility of~$E_{22}$, and the property $\rk_\C \lambda E_{33} - A_{33} = n_3$ for all $\lambda\in\C$ is similar to the proof of~\cite[Thm.~3.3]{BergTren14} and omitted.
\end{Proof}

We revisit again Example~\ref{Ex:P-form} to illustrate how Theorem~\ref{thm:QPDFF} can be utilized to obtain a QPDFF.

\begin{Example}[Example~\ref{Ex:P-form} revisited]
  Consider again $[E,A,B]\in\Sigma_{7,6,3}$ from Example~\ref{Ex:P-form}. The augmented Wong-sequences have been calculated in Example~\ref{Ex:QPFF} and we can choose $T$ and $S$ as
  \[ T = [U_T,R_T,O_T] = \left[\begin{smallarray}{ccc|c|cc} -8 & -5 & 2 & 7 & -1 & 1\\ 4 & 1 & -2 & -5 & -1 & 1\\ -2 & -1 & 0 & 1 & 1 & -1\\ 1 & 0 & 0 & 0 & 0 & -1\\ 0 & 1 & 0 & 1 & 0 & 0\\ 0 & 0 & 1 & 1 & 1 & 0
  \end{smallarray}\right],\quad
     S^{-1} = [U_S,R_S,O_S,Q_S] = \left[\begin{smallarray}{c|c|cc|ccc}
        0 & 1 & -1 & -1 & 1 & 0 & 0\\ 1 & -1 & 0 & 1 & 0 & 1 & 0\\ 0 & 0 & 1 & 0 & 0 & 0 & 1\\ 1 & 1 & 1 & 1 & 1 & 1 & 0\\ 1 & 0 & 0 & 0 & 0 & 1 & 0\\ 0 & 1 & 1 & 1 & -1 & -\frac{5}{2} & -2\\ 1 & 3 & 1 & 1 & 1 & -4 & -4
     \end{smallarray}\right].
  \]
  Based on this choice for $T$ and $S$, we can furthermore choose
  \[
     F_P = \begin{smallbmatrix}
        -7 & -9 & 0 & 5 & -\frac{37}{2} & \frac{37}{2}\\ -\frac{42}{5} & -\frac{34}{5} & \frac{4}{5} & \frac{28}{5} & -9 & 9\\ \frac{22}{5} & \frac{14}{5} & -\frac{9}{5} & -\frac{13}{5} & -1 & -2
     \end{smallbmatrix},\quad
     F_D = \begin{smallbmatrix}
         -7 & -6 & 4 & 16 & -10 & 9\\ -\frac{18}{5} & -3 & \frac{11}{5} & \frac{36}{5} & -\frac{13}{5} & 2\\ -\frac{2}{5} & 0 & \frac{4}{5} & \frac{4}{5} & -\frac{7}{5} & -1
     \end{smallbmatrix},\quad
     V = I_3.
     \]
  This results in the following QPDFF
  \[
      [S(ET+BF_D),S(AT+BF_P),SBV] = \left[
         \left[\begin{smallarray}{ccc|c|cc}
            \frac{1}{5} & 0 & -\frac{2}{5} & \frac{8}{5} & -\frac{24}{5} & 5\\\hline
            0 & 0 & 0 & 2 & -4 & 4\\ \hline
            0 & 0 & 0 & 0 & 0 & 0\\
            0 & 0 & 0 & 0 & 0 & 0\\\hline
            0 & 0 & 0 & 0 & 0 & 0\\ 0 & 0 & 0 & 0 & 0 & 0\\ 0 & 0 & 0 & 0 & 0 & 0
         \end{smallarray}\right],
         \left[\begin{smallarray}{ccc|c|cc}
            \frac{4}{5} & \frac{3}{5} & -\frac{3}{5} & \frac{4}{5} & 0 & -1\\\hline
            0 & 0 & 0 & 2 & -3 & 1\\\hline
            0 & 0 & 0 & 0 & \frac{13}{2} & \frac{1}{2}\\
            0 & 0 & 0 & 0 & -8 & 0\\\hline
            0 & 0 & 0 & 0 & 0 & 0\\ 0 & 0 & 0 & 0 & 0 & 0\\ 0 & 0 & 0 & 0 & 0 & 0
         \end{smallarray}
         \right],
         \begin{smallbmatrix}
            0 & 0 & 0\\\hline
            0 & 0 & 0\\\hline
            0 & 0 & 0\\ 0 & 0 & 0\\\hline
            1 & -1 & -1\\ 0 & 0 & 2\\ -1 & 2 & -3
         \end{smallbmatrix}
      \right].
  \]
\end{Example}

\begin{Proposition}[Uniqueness of~QPDFF]\label{Prop:QPDFF-unique}
Let $[E,A,B]\in\Sigma_{\ell,n,m}$ and $S_1,S_2\in\Gl_{\ell}(\R)$, $T_1,T_2\in\Gl_{n}(\R)$, $V_1, V_2\in\Gl_m(\R)$, $F_P^1, F_P^2, F_D^1, F_D^2\in\R^{m\times n}$ be such that, for $i=1,2$,
\[
   [E, A , B] \overset{S_i,T_i,V_i,F_P^i,F_D^i}{\cong_{PD}} [E_i, A_i, B_i] = \left[\begin{bmatrix} E_{11,i} & E_{12,i} & E_{13,i} \\ 0 & E_{22,i} & E_{23,i} \\ 0 & 0 & E_{33,i} \\ 0&0&0\end{bmatrix}, \begin{bmatrix} A_{11,i} & A_{12,i} & A_{13,i} \\ 0 & A_{22,i} & A_{23,i} \\ 0 & 0 & A_{33,i} \\ 0&0&0\end{bmatrix}, \begin{bmatrix} 0&0\\ 0&0\\ 0&0\\ 0&\hat B_i \end{bmatrix}\right],
\]
where $[E_i,A_i,B_i]$ is in~QPDFF~\eqref{eq:QPDFF}. Then the corresponding diagonal blocks (which have matching sizes according to Lemma~\ref{Lem:Wong-QPDFF}) are equivalent in the sense that there exist invertible matrices~$S_{ii}$, $i=1,2,3,4$, $T_{ii}$, $i=1,2,3$, and $V_{22}$ such that
\[
    E_{ii,1} = S_{ii} E_{ii,2} T_{ii},\quad i=1,2,3, \quad\text{and}\quad \widehat{B}_1 = S_{44} \widehat{B}_2 V_{22}.
\]
%with corresponding block sizes given by $\ell_{1,i}$, $n_{1,i}$, $\ell_{2,i}$, $n_{2,i}$, $\ell_{3,i}$, $n_{3,i}$, $m_{1,i}$, $m_{2,i}$.\\
%Then $\ell_{1,1}=\ell_{1,2}, \ell_{2,1}=\ell_{2,2}, \ell_{3,1}=\ell_{3,1}, n_{1,1}=n_{1,2}, n_{2,1}=n_{2,2}, n_{3,1}=n_{3,2}$, $m_{1,1} = m_{1,2}$, $m_{2,1} = m_{2,2}$ and, moreover, for some $S_{11}\in\Gl_{\ell_{1,1}}(\R), S_{22}\in\Gl_{\ell_{2,1}}(\R), S_{33}\in\Gl_{\ell_{4,1}}(\R), S_{44}\in\Gl_{m_{2,1}}(\R), T_{11}\in\Gl_{n_{1,1}}(\R), T_{22}\in\Gl_{n_{2,1}}(\R), T_{33}\in\Gl_{n_{3,1}}(\R), V_{11}\in \Gl_{m_{1,1}}(\R), V_{22}\in\Gl_{m_{2,1}}(\R)$ and $S_{12}, S_{13}, S_{14}, S_{23}, S_{24}, S_{34}$, $T_{12}, T_{13}, T_{23}, V_{12}$ of appropriate sizes we have that
%\[
%    S_2 S_1^{-1} = \begin{bmatrix} S_{11} & S_{12} & S_{13} & 0\\ 0 & S_{22} & S_{23} & 0\\0&0& S_{33} & 0\\ S_{41}&S_{42}&S_{43}&S_{44}\end{bmatrix},\quad T_1^{-1} T_2 = \begin{bmatrix} T_{11} & T_{12} & T_{13}\\ 0 & T_{22} & T_{23}\\0&0& T_{33}\end{bmatrix},\quad V_1^{-1} V_2 = \begin{bmatrix} V_{11} & V_{12}\\ 0&V_{22}\end{bmatrix}.
%\]
\end{Proposition}

\begin{proof}
Without loss of generality we assume that $S_1=I_\ell$, $T_1=I_n$, $V_1 = I_m$ and $F_P^1 = F_D^1 = 0$.

\noindent
\emph{Step 1}: By Lemma~\ref{lem:Wong-PDfb} we have
%that
\[
    \cV^*_{[E_1,A_1,B_1]} = T_2 \cV^*_{[E_2,A_2,B_2]},\quad \cW^*_{[E_1,A_1,B_1]} = T_2 \cW^*_{[E_2,A_2,B_2]},
\]
and from Lemma~\ref{Lem:Wong-QPDFF} we obtain
\[
    \cV^*_{[E_{i}, A_{i}, B_{i}]}\cap \cW^*_{[E_{i}, A_{i}, B_{i}]} = \R^{n_{1,i}}\times \{0\}^{n_{2,i}+n_{3,i}}
    \qquad
    \text{for $i=1,2$.}
\]
This implies $n_{1,1}=n_{1,2}$  and
\[
    T_2 = \begin{bmatrix} T_{11} &T_{12} &T_{13}\\ 0 & T_{22}&T_{23}\\ 0&T_{32}&T_{33}\end{bmatrix}\quad\text{for}\ T_{11}\in\Gl_{n_{1,1}}(\R), T_{22}\in\R^{n_{2,1}\times n_{2,2}}, T_{33}\in\R^{n_{3,1}\times n_{3,2}}
\]
and $T_{12}, T_{13}, T_{23}, T_{32}$ of appropriate sizes. Furthermore, Lemma~\ref{Lem:Wong-QPDFF} gives
\[
    \R^{n_{1,1}+n_{2,1}}\times\{0\}^{n_{3,1}} = \cV^*_{[E_1,A_1,B_1]} = T_2 \cV^*_{[E_2,A_2,B_2]} = T_2 (\R^{n_{1,2}+n_{2,2}}\times \{0\}^{n_{3,2}}),
\]
which, together with $n_{1,1}=n_{1,2}$, yields
$n_{2,1}=n_{2,2}$, $n_{3,1}=n_{3,2}$ and
\[
    T_{32}=0, \quad T_{22}\in\Gl_{n_{2,1}}(\R),\quad T_{33}\in\Gl_{n_{3,1}}(\R).
\]

\noindent
\emph{Step 2}: Partitioning
\[
    S_2 = \begin{bmatrix} S_{11} & S_{12} & S_{13} & S_{14}\\ S_{21} & S_{22} & S_{23} & S_{24}\\ S_{31} &S_{32}& S_{33} & S_{34}\\S_{41}&S_{42}&S_{43}&S_{44}\end{bmatrix}\quad\text{for}\ S_{11}\in\R^{\ell_{1,2}\times \ell_{1,1}}, S_{22}\in\R^{\ell_{2,2}\times \ell_{2,1}}, S_{33}\in\R^{\ell_{3,2}\times \ell_{3,1}}, S_{44}\in\R^{m_{2,2}\times m_{2,1}}
\]
and
\[
    V_2 = \begin{bmatrix} V_{11} & V_{12}\\ V_{21}&V_{22}\end{bmatrix}\quad\text{for}\ V_{11}\in\R^{m_{1,2}\times m_{1,1}}, V_{22}\in\R^{m_{2,2}\times m_{2,1}}
\]
and off-diagonal block matrices of appropriate sizes, the equation
 $S_2B_1 V_2=B_2$ in conjunction with
 $\hat S_4 := [S_{14}^\top, S_{24}^\top, S_{34}^\top]^\top$
 gives
\[
    \hat S_4 \hat B_1 V_{21} = 0,\ \hat S_4 \hat B_1 V_{22} = 0,\ S_{44} \hat B_1 V_{21} = 0,\ S_{44} \hat B_1 V_{22} = \hat B_2.
\]
Since $\hat B_1 [V_{21}, V_{22}]$ has full row rank, it follows that $S_{14}=0, S_{24}=0$ and $S_{34}=0$.  By Definition~\ref{def:QPDFF} we have $m_{2,1} = \rk B = m_{2,2}$ and hence also $m_{1,1} = m_{1,2}$. Since $S_2$ is invertible, this implies that $S_{44}\in\Gl_{m_{2,1}}(\R)$, thus $V_{21} = 0$, $V_{22}\in\Gl_{m_{2,1}}(\R)$ and, since $V_2$ is invertible, $V_{11}\in\Gl_{m_{1,1}}(\R)$.

The equation $S_2(E_1T_2 + B_1 F_P^2)=E_2$ yields
$\begin{smallbmatrix} S_{21}\\ S_{31}\end{smallbmatrix} E_{11,1} T_{11} = 0$ and the full row rank of~$E_{11,1}$ implies~$S_{21}=0$ and~$S_{31}=0$.
Since~$S_2$ is invertible, it follows that $\ell_{1,1}\le \ell_{1,2}$.
Reversing the roles of~$[E_1,A_1,B_1]$ and~~
$[E_2,A_2,B_2]$ gives $\ell_{1,1}\geq \ell_{1,2}$, whence~$\ell_{1,1}=\ell_{1,2}$. We further have the equation
\[
    S_{32} E_{22,1} T_{22} = 0
\]
which by invertibility of~$T_{22}$ and~$E_{22,1}$ gives that~$S_{32}=0$.
This finally implies $\ell_{2,1}=\ell_{2,2}=n_{2,1}=n_{2,2}$, $\ell_{3,1}=\ell_{3,2}$, $S_{22}\in\Gl_{\ell_{2,1}}(\R)$,  $S_{33}\in\Gl_{\ell_{3,1}}(\R)$,
and  hence the proof of the proposition is complete.
\end{proof}

\begin{Remark}[Geometric interpretation of QPDFF]
   Since the state-space transformations $T$ for the QPDFF and the QPFF can be chosen to be identical (the same subspaces are used to define it), it follows that the corresponding augmented Wong limits have the same control-theoretic geometric interpretation. The key geometric difference between the QPDFF and the QPFF is the reinterpretations of states in the QPFF which are completely controllable due to the external input (represented by the subspace $\im B$) as underdetermined (and hence completely controllable) variables in the QPDFF.
\end{Remark}

\section{Conclusion}

We have presented the novel concepts of quasi P-feedback and quasi PD-feedback forms for~DAE control systems which reveal the key structural properties of the control system under P(D)-feedback transformations. Furthermore, the forms are easily obtained via the augmented Wong-sequences, which additionally provides a geometric insight.

\printbibliography[notcategory=bibexclude]

\appendix

\section{Linear matrix equations}\label{app:ProofDecoupledQPFF}

In order to prove the claim of Proposition~\ref{Prop:decoupledQPFF} we recall two important results concerning the solvability of linear matrix equations.

\begin{Lemma}[{\cite[Lem.~4.14]{BergTren12}}, c.f.~\cite{Chu87}]\label{lem:twoEqs2genSylvester}
   Consider the the two matrix equations
   \begin{equation}\label{eq:twoMatrixEqs}
   \begin{aligned}
       0 &= \mathrm{E} + \mathrm{A} Y + Z \mathrm{D} \\
       0 &= \mathrm{F} + \mathrm{C} Y + Z \mathrm{B}
     \end{aligned}
   \end{equation}
   for some $\mathrm{A},\mathrm{C}\in\R^{m\times n}$, $\mathrm{B},\mathrm{D}\in\R^{p\times q}$, $\mathrm{E},\mathrm{F}\in\R^{m\times q}$. Assume that the pencil $s\mathrm{B}-\mathrm{D}$ has full polynomial column rank, i.e.\ there exists $\lambda\in\R$ such that $(\lambda \mathrm{B} - \mathrm{D})$ has a left inverse $(\lambda \mathrm{B} - \mathrm{D})^\dagger$. Then~\eqref{eq:twoMatrixEqs} is solvable, if the generalized Sylvester equation
   \[
       \mathrm{A} X \mathrm{B} - \mathrm{C} X \mathrm{D} = - \mathrm{E} + (\lambda \mathrm{E} - \mathrm{F}) (\lambda \mathrm{B} - \mathrm{D})^\dagger \mathrm{D}
   \]
   is solvable. \qed
\end{Lemma}

\begin{Remark}\label{rem:transposeTwoMatrixEqs}
   By considering the transposed version of~\eqref{eq:twoMatrixEqs},
    the same solvability reduction to a generalized Sylvester equation is feasible, provided that the matrix pencil $s\mathrm{C}-\mathrm{A}$ has full polynomial row rank. In fact, in this case~\eqref{eq:twoMatrixEqs} is solvable, if
   \[
       \mathrm{A} X \mathrm{B} - \mathrm{C} X \mathrm{D} = - \mathrm{F} + \mathrm{C} (\lambda \mathrm{C} - \mathrm{A})^\dagger (\lambda \mathrm{F} - \mathrm{E})
   \]
   is solvable.
\end{Remark}

\begin{Lemma}[{\cite[Thm.~2]{HernGass89}}]\label{lem:Sylvester-equation}
   Consider the generalized Sylvester equation
   \begin{equation}\label{eq:genSylvester}
       \mathrm{A} X \mathrm{B} - \mathrm{C} X \mathrm{D} = \mathrm{E}
   \end{equation}
   for some matrices $\mathrm{A},\mathrm{C}\in\R^{m\times n}$, $\mathrm{B},\mathrm{D}\in\R^{p\times q}$, $\mathrm{E}\in\R^{m\times q}$. Assume that the matrix pencils $s\mathrm{C}-\mathrm{A}$ and $s\mathrm{B}-\mathrm{D}$ have full polynomial rank and assume furthermore that there is no $\lambda\in\C\cup\{\infty\}$ such that both matrices $\lambda \mathrm{C} - \mathrm{A}$ and $\lambda \mathrm{B} - \mathrm{D}$ have a simultaneous rank drop\footnote{Here we use the convention that $\rank(\infty \mathrm{M} - \mathrm{N}) = \rank\mathrm{M}.$}. Then the generalized Sylvester equation~\eqref{eq:genSylvester} has a solution $X\in\R^{n\times p}.$\qed
\end{Lemma}

%%%%%%%%%%%%%%%%%%%%%%%%%%%%%%%%%%%
%References
%%%%%%%%%%%%%%%%%%%%%%%%%%%%%%%%%%%

%\printbibliography[notcategory=bibexclude]

\end{document}

%% file: MPtopbook.tex
\usepackage{comment} %for easily switching on and of Exercises and/or their solutions
\usepackage{mathtools} % includes everything amsmath has, but fixes some errors and has some additional beneficts!
\usepackage{amsthm}
\usepackage{amssymb}

\usepackage{bbm} % for \mathbbm{1}
\usepackage{amsxtra}
\usepackage{amsfonts}
\usepackage{latexsym}
\usepackage[mathscr]{eucal}
  \makeindex

\usepackage{cancel} %  AI 21.2.2021
\usepackage{multicol}
\usepackage{mathdots}%Aenderung Fabian fuer
\usepackage[shortlabels]{enumitem} %Aktuelle und umfangreichere Version von enumerate
  \setlist[enumerate]{leftmargin=*}
  \setlist[enumerate,1]{label={\textup{(\roman*)}}}
  \setlist[description]{font=\normalfont}
\usepackage{calc} % for widthof

\usepackage{nicefrac} %fracs in the form 1/2

\usepackage{subcaption} % more recent version of subfig and subfigure

\usepackage[hidelinks]{hyperref}

\usepackage[
    backend=bibtex,
    style=numeric-comp,%authoryear-comp,
    sorting=nyt,
    firstinits=true,
    doi=false,url=false,isbn=false,
    maxbibnames=99
 ]{biblatex}
\usepackage{pdflscape}

\usepackage{longtable} % must be loaded before arydshln
\usepackage{tabu} % for using longtabu with X columns
\makeatletter
\renewcommand*\env@matrix[1][*\c@MaxMatrixCols c]{%
  \hskip -\arraycolsep
  \let\@ifnextchar\new@ifnextchar
  \array{#1}}
\makeatother
\usepackage{listings} %Aenderung Stephan fuer bessere Darstellung von Matlab code
\usepackage{tikz}%Aenderung Stefan fuer Bilder mit tikz 23.06.2010
\usetikzlibrary{fit}%Aenderung Stefan fuer Bilder mit tikz 23.06.2010
\usetikzlibrary{calc}%Aenderung Stefan fuer Bilder mit tikz 23.06.2010
\pgfdeclarelayer{background}%Aenderung Stefan fuer Bilder mit tikz 23.06.2010
\pgfsetlayers{background,main}%Aenderung Stefan fuer Bilder mit tikz 23.06.2010
\colorlet{inbox}{lightgray!20}
\colorlet{outbox}{lightgray!50}

\usepackage{thmtools}
\declaretheoremstyle[headfont=\normalfont\bfseries]{bfthmstyle}
\declaretheorem[numberwithin=section,style=bfthmstyle]{Theorem}
\declaretheorem[sharenumber=Theorem,style=bfthmstyle]{Lemma}
\declaretheorem[sharenumber=Theorem,style=bfthmstyle]{Proposition}

\declaretheoremstyle[headfont=\normalfont\bfseries,qed={$\diamond$}]{bfdefstyle}
\declaretheorem[sharenumber=Theorem,style=bfdefstyle]{Definition}
\declaretheorem[sharenumber=Theorem,style=bfdefstyle]{Example}

\declaretheoremstyle[headfont=\normalfont\bfseries,qed={$\diamond$}]{bfremstyle}
\declaretheorem[sharenumber=Theorem,style=bfdefstyle]{Remark}

\makeatletter
\newenvironment{Proof}[1][\proofname]{\par
  \pushQED{\hfill$\square$}%
  \normalfont \topsep6\p@\@plus6\p@\relax
  \trivlist
  \item[\hskip\labelsep
        \itshape
    #1\@addpunct{.}]\ignorespaces
}{%
  \popQED\endtrivlist\@endpefalse
}
\makeatother

\newcommand{\mb}[1]{\mathbb{#1}}

\DeclareMathOperator{\rank}{rank}

\newcommand{\R}{\mb{R}}

\newcommand{\C}{\mb{C}}

\newcommand{\N}{{\mb{N}}}

\newcommand{\cV}{\mathcal{V}}
\newcommand{\cW}{\mathcal{W}}

\newcommand{\bsalpha}{{\boldsymbol{\alpha}}}

\newcommand{\bsdelta}{{\boldsymbol{\delta}}}
\newcommand{\bsbeta}{{\boldsymbol{\beta}}}
\newcommand{\bsgamma}{{\boldsymbol{\gamma}}}
\newcommand{\bskappa}{{\boldsymbol{\kappa}}}

\newcommand{\setdef}[2]{\left\{\ #1\ \left|\ \vphantom{#1} #2\ \right.\right\}}

\newcommand{\ddt}{\tfrac{\text{\normalfont d}}{\text{\normalfont d}t}}

\renewcommand{\imath}{\mathrm{i}}

\DeclareMathOperator{\diag}{diag}
\DeclareMathOperator{\rk}{rk}
\DeclareMathOperator{\im}{im}

{\left(\begin{smallmatrix}}%            begin code
{\end{smallmatrix}\right)}%             end code

\newenvironment{smallbmatrix}%          environment name
{\left[\begin{smallmatrix}}%            begin code
{\end{smallmatrix}\right]}%             end code

\newcommand{\Gl}{\GL}
\newcommand{\GL}{GL}

\newlength{\innersep}
\newlength{\maxlength}
\newlength{\dummylength}

\newcommand{\JordanBlock}[3]{
\setlength{\arraycolsep}{0pt}
\renewcommand{\arraystretch}{0}
\settowidth{\maxlength}{$#1$}
\settoheight{\dummylength}{$#1$}
\ifdim\dummylength>\maxlength
  \setlength{\maxlength}{\dummylength}
\fi
\settowidth{\dummylength}{$#2$}
\ifdim\dummylength>\maxlength
  \setlength{\maxlength}{\dummylength}
\fi
\settoheight{\dummylength}{$#2$}
\ifdim\dummylength>\maxlength
  \setlength{\maxlength}{\dummylength}
\fi
\setlength{\innersep}{0.1\maxlength}
\addtolength{\maxlength}{\innersep}
\addtolength{\maxlength}{\innersep}
\newcommand{\invisiblebox}{\phantom{\rule{\maxlength}{\maxlength}}}
\begin{array}{cccc}
  \tikz[remember picture] \node[outer sep=0,inner sep=\innersep] (a11) {$#1$}; &{\tikz[remember picture] \node[outer sep=0,inner sep=\innersep] (a21) {$#2$};}& & \invisiblebox\\
  &&\phantom{\rule{#3}{#3}} &\\
   \invisiblebox &\invisiblebox& &{\tikz[remember picture] \node[outer sep=0,inner sep=\innersep] (a43) {$#2$};} \\
   \invisiblebox &\invisiblebox& &
   {\tikz[remember picture] \node[outer sep=0,inner sep=\innersep] (a44) {$#1$};}
\end{array}
\tikz[remember picture, overlay] \draw (a11) edge[very thick] (a44);
\tikz[remember picture, overlay] \draw (a21) edge[very thick] (a43);
}

\newcommand{\LowerNilBlock}[3]{
\setlength{\arraycolsep}{0pt}
\renewcommand{\arraystretch}{0}
\settowidth{\maxlength}{$#1$}
\settoheight{\dummylength}{$#1$}
\ifdim\dummylength>\maxlength
  \setlength{\maxlength}{\dummylength}
\fi
\settowidth{\dummylength}{$#2$}
\ifdim\dummylength>\maxlength
  \setlength{\maxlength}{\dummylength}
\fi
\settoheight{\dummylength}{$#2$}
\ifdim\dummylength>\maxlength
  \setlength{\maxlength}{\dummylength}
\fi
\setlength{\innersep}{0.1\maxlength}
\addtolength{\maxlength}{\innersep}
\addtolength{\maxlength}{\innersep}
\newcommand{\invisiblebox}{\phantom{\rule{\maxlength}{\maxlength}}}
\begin{array}{cccc}
  \tikz[remember picture] \node[outer sep=0,inner sep=\innersep] (a11) {$#1$}; &  & \invisiblebox & \invisiblebox\\
  {\tikz[remember picture] \node[outer sep=0,inner sep=\innersep] (a21) {$#2$};}&  & \invisiblebox & \invisiblebox\\
   & \phantom{\rule{#3}{#3}} &  & \\
   \invisiblebox & &  {\tikz[remember picture] \node[outer sep=0,inner sep=\innersep] (a43) {$#2$};} &
   {\tikz[remember picture] \node[outer sep=0,inner sep=\innersep] (a44) {$#1$};}
\end{array}
\tikz[remember picture, overlay] \draw (a11) edge[very thick] (a44);
\tikz[remember picture, overlay] \draw (a21) edge[very thick] (a43);
}

\newcommand{\UpperNilBlock}[3]{
\setlength{\arraycolsep}{0pt}
\renewcommand{\arraystretch}{0}
\settowidth{\maxlength}{$#1$}
\settoheight{\dummylength}{$#1$}
\ifdim\dummylength>\maxlength
  \setlength{\maxlength}{\dummylength}
\fi
\settowidth{\dummylength}{$#2$}
\ifdim\dummylength>\maxlength
  \setlength{\maxlength}{\dummylength}
\fi
\settoheight{\dummylength}{$#2$}
\ifdim\dummylength>\maxlength
  \setlength{\maxlength}{\dummylength}
\fi
\setlength{\innersep}{0.1\maxlength}
\addtolength{\maxlength}{\innersep}
\addtolength{\maxlength}{\innersep}
\newcommand{\invisiblebox}{\phantom{\rule{\maxlength}{\maxlength}}}
\begin{array}{cccc}
  \tikz[remember picture]{\node[outer sep=0,inner sep=\innersep] (a11) {$#1$};} & \tikz[remember picture]{\node[outer sep=0,inner sep=\innersep] (a21) {$#2$};}  &  & \invisiblebox\\
  &  & \phantom{\rule{#3}{#3}} &  \\
  \invisiblebox & \invisiblebox & & \tikz[remember picture]{\node[outer sep=0,inner sep=\innersep] (a43) {$#2$};}\\
   \invisiblebox & \invisiblebox &   & \tikz[remember picture]{\node[outer sep=0,inner sep=\innersep] (a44) {$#1$};}
\end{array}
\tikz[remember picture, overlay] \draw (a11) edge[very thick] (a44);
\tikz[remember picture, overlay] \draw (a21) edge[very thick] (a43);
}

\newcommand{\SmallNilBlock}[3]{
\LowerNilBlock{\mbox{\scriptsize $#1$}}{\mbox{\scriptsize $#2$}}{#3}
}

\newcommand{\RectBlock}[3]{
\setlength{\arraycolsep}{0pt}
\renewcommand{\arraystretch}{0}
\settowidth{\maxlength}{$#1$}
\settoheight{\dummylength}{$#1$}
\ifdim\dummylength>\maxlength
  \setlength{\maxlength}{\dummylength}
\fi
\settowidth{\dummylength}{$#2$}
\ifdim\dummylength>\maxlength
  \setlength{\maxlength}{\dummylength}
\fi
\settoheight{\dummylength}{$#2$}
\ifdim\dummylength>\maxlength
  \setlength{\maxlength}{\dummylength}
\fi
\setlength{\innersep}{0.1\maxlength}
\addtolength{\maxlength}{\innersep}
\addtolength{\maxlength}{\innersep}
\newcommand{\invisiblebox}{\phantom{\rule{\maxlength}{\maxlength}}}
\begin{array}{ccccc}
  \tikz[remember picture] \node[outer sep=0,inner sep=\innersep] (a11) {$#1$}; &{\tikz[remember picture] \node[outer sep=0,inner sep=\innersep] (a21) {$#2$};}& & \invisiblebox & \invisiblebox\\
  &&\phantom{\rule{#3}{#3}} &&\\
   \invisiblebox &\invisiblebox& &
   {\tikz[remember picture] \node[outer sep=0,inner sep=\innersep] (a44) {$#1$};} & {\tikz[remember picture] \node[outer sep=0,inner sep=\innersep] (a43) {$#2$};}
\end{array}
\tikz[remember picture, overlay] \draw (a11) edge[very thick] (a44);
\tikz[remember picture, overlay] \draw (a21) edge[very thick] (a43);
}

\newcommand{\RectBlockT}[3]{
\setlength{\arraycolsep}{0pt}
\renewcommand{\arraystretch}{0}
\settowidth{\maxlength}{$#1$}
\settoheight{\dummylength}{$#1$}
\ifdim\dummylength>\maxlength
  \setlength{\maxlength}{\dummylength}
\fi
\settowidth{\dummylength}{$#2$}
\ifdim\dummylength>\maxlength
  \setlength{\maxlength}{\dummylength}
\fi
\settoheight{\dummylength}{$#2$}
\ifdim\dummylength>\maxlength
  \setlength{\maxlength}{\dummylength}
\fi
\setlength{\innersep}{0.1\maxlength}
\addtolength{\maxlength}{\innersep}
\addtolength{\maxlength}{\innersep}
\newcommand{\invisiblebox}{\phantom{\rule{\maxlength}{\maxlength}}}
\begin{array}{ccccc}
  \tikz[remember picture]{\node[outer sep=0,inner sep=\innersep] (a11) {$#1$};} & & \invisiblebox\\
  \tikz[remember picture]{\node[outer sep=0,inner sep=\innersep] (a21) {$#2$};} & & \invisiblebox\\
  &\phantom{\rule{#3}{#3}} &\\
   \invisiblebox & & \tikz[remember picture]{\node[outer sep=0,inner sep=\innersep] (a44) {$#1$};}\\
   \invisiblebox & & \tikz[remember picture]{\node[outer sep=0,inner sep=\innersep] (a43) {$#2$};}
\end{array}
\tikz[remember picture, overlay] \draw (a11) edge[very thick] (a44);
\tikz[remember picture, overlay] \draw (a21) edge[very thick] (a43);
}

\newcommand{\SmallRectBlock}[3]{
\RectBlock{\mbox{\scriptsize $#1$}}{\mbox{\scriptsize $#2$}}{#3}
}

\newcommand{\QBlock}[3]{
\setlength{\arraycolsep}{0pt}
\renewcommand{\arraystretch}{0}
\settowidth{\maxlength}{$#1$}
\settoheight{\dummylength}{$#1$}
\ifdim\dummylength>\maxlength
  \setlength{\maxlength}{\dummylength}
\fi
\settowidth{\dummylength}{$#2$}
\ifdim\dummylength>\maxlength
  \setlength{\maxlength}{\dummylength}
\fi
\settoheight{\dummylength}{$#2$}
\ifdim\dummylength>\maxlength
  \setlength{\maxlength}{\dummylength}
\fi
\setlength{\innersep}{0.1\maxlength}
\addtolength{\maxlength}{\innersep}
\addtolength{\maxlength}{\innersep}
\newcommand{\invisiblebox}{\phantom{\rule{\maxlength}{\maxlength}}}
\begin{array}{ccc}
  \tikz[remember picture] \node[outer sep=0,inner sep=\innersep] (a11) {$#1$}; &   & \invisiblebox\\
  {\tikz[remember picture] \node[outer sep=0,inner sep=\innersep] (a21) {$#2$};}&   & \invisiblebox\\
   & \phantom{\rule{#3}{#3}} &  \\
   \invisiblebox & &
   {\tikz[remember picture] \node[outer sep=0,inner sep=\innersep] (a44) {$#1$};}\\
     \invisiblebox & & {\tikz[remember picture] \node[outer sep=0,inner sep=\innersep] (a43) {$#2$};}
\end{array}
\tikz[remember picture, overlay] \draw (a11) edge[very thick] (a44);
\tikz[remember picture, overlay] \draw (a21) edge[very thick] (a43);
}

\usepackage{afterpage}% for "\afterpage"
\usepackage{xcolor}
\usepackage{pagecolor}
\definecolor{carolinablue}{rgb}{0.6, 0.73, 0.89}
\definecolor{bluegray}{rgb}{0.4, 0.6, 0.8}
\definecolor{darkpastelblue}{rgb}{0.47, 0.62, 0.8}        %  120, 158, 204
\definecolor{glaucous}{rgb}{0.38, 0.51, 0.71}
\definecolor{almond}{rgb}{0.94, 0.87, 0.8}
\definecolor{beige}{rgb}{0.96, 0.96, 0.86}
\definecolor{darkchampagne}{rgb}{0.76, 0.7, 0.5}
\definecolor{champagne}{rgb}{0.97, 0.91, 0.81}
\definecolor{wheat}{rgb}{0.96, 0.87, 0.7}
\definecolor{almond}{rgb}{0.94, 0.87, 0.8}
\definecolor{blanchedalmond}{rgb}{1.0, 0.92, 0.8}